\documentclass[11pt,a4paper]{article}

\usepackage{epsfig}
\usepackage{amssymb}
\usepackage{amsmath}
\usepackage{amsthm}
\usepackage{empheq}

\usepackage{graphics,graphicx}
\usepackage{eepic}
\usepackage[Symbolsmallscale]{upgreek}
\usepackage{mathrsfs}
\usepackage[active]{srcltx}
\usepackage{color}
\usepackage{epsfig}
\usepackage{dsfont}
\usepackage{caption}
\usepackage{subcaption}
\usepackage{amscd}
\usepackage{xargs}
\usepackage{enumerate}
\usepackage[usenames,dvipsnames]{xcolor}
\usepackage[colorlinks=true,linkcolor=blue,urlcolor=blue,filecolor=blue, citecolor=blue]{hyperref}
\usepackage{GGGraphics}

\newtheorem{theo}{Theorem}[section]
\newtheorem{theorem}[theo]{Theorem}
\newtheorem{lemma}[theo]{Lemma}
\newtheorem{proposition}[theo]{Proposition}

\newtheorem{remark}[theo]{Remark}

\newtheorem{example}[theo]{Example}

\def\qed{\hfill $\square$ \goodbreak \bigskip}

\newtheorem{assumption}{Assumption}

\newcommand{\be}{\begin{equation}}
\newcommand{\ee}{\end{equation}}
\newcommand{\bes}{\begin{equation*}}
\newcommand{\ees}{\end{equation*}}
\newcommand{\bea}{\begin{eqnarray}}
\newcommand{\eea}{\end{eqnarray}}

\newcommand{\R}{\mathbb{R}}
\newcommand{\IN}{\mathbb{N}}

\def\exponentTheo{\chi}

\def \IR{\mathbb R}
\def \IN{\mathbb N}

\def \IE{\mathbb E}

  \newcommand{\defEns}[1]{\left\lbrace #1 \right\rbrace }

\def\bigO{\mathcal{O}}
\def\Repsilon{R_{\epsilon}}
\def\Rqtransition{R_q}
\def\iid{i.i.d.}
\newcommandx\sequence[3][2=,3=]
{\ifthenelse{\equal{#3}{}}{\ensuremath{\{ #1_{#2}\}}}{\ensuremath{\{ #1_{#2}, \eqsp #2 \in #3 \}}}}
\newcommandx{\sequencen}[2][2=n\in\N]{\ensuremath{\{ #1, \eqsp #2 \}}}
\newcommandx\sequenceDouble[4][3=,4=]
{\ifthenelse{\equal{#3}{}}{\ensuremath{\{ (#1_{#3},#2_{#3}) \}}}{\ensuremath{\{  (#1_{#3},#2_{#3}), \eqsp #3 \in #4 \}}}}
\newcommandx{\sequencenDouble}[3][3=n\in\N]{\ensuremath{\{ (#1_{n},#2_{n}), \eqsp #3 \}}}
\def\Varfun{\mathrm{V}}
\def\rme{e}
\def\rmd{\mathrm{d}}
\def\Sp{\operatorname{Sp}}
\def\Id{\operatorname{I}_d}
\newcommand{\plusinfty}{+\infty}
\def\ie{\textit{i.e.}}
\def\eqsp{\;}
\newcommand{\abs}[1]{\left\vert #1 \right\vert}
\newcommand{\norm}[1]{\left\Vert #1 \right \Vert}
\def\Mrm{\mathrm{M}}
\def\Arm{\mathrm{A}}

\newcommand{\lambdaMin}[1]{\lambda_{\min}\left(#1\right)}
\newcommand{\lambdaMax}[1]{\lambda_{\max}\left(#1\right)}

\newcommand{\coint}[1]{\left[#1\right)}
\newcommand{\ocint}[1]{\left(#1\right]}
\newcommand{\ooint}[1]{\left(#1\right)}
\newcommand{\ccint}[1]{\left[#1\right]}
\newcommandx\lapGrad[2][1=f,2=\varx]{\overrightarrow{\Delta}  #1 (#2)}
\newcommand{\1}{\mathds{1}}
\def\Escr{\mathscr{E}}
\def\Pscr{\mathscr{P}}
\def\Ecal{\mathcal{E}}
\def\qrm{\mathrm{q}}
\def\rrm{\mathrm{r}}
\def\prm{\mathrm{p}}
\def\kbf{\mathbf{k}}
\def\Var{\operatorname{Var}}
\newcommand{\ps}[2]{\left \langle#1,#2 \right\rangle}
\def\Leb{\operatorname{Leb}}
\def\smallSet{\mathcal{C}}
\def\xrm{\mathrm{X}}
\def\Prm{\mathrm{P}}
\newcommand{\proba}[1]{\mathbb{P}\left[ #1 \right]}
\def\fMALA{\operatorname{fM}}
\def\mMALA{\operatorname{fM}}
\def\mO{\operatorname{mO}}
\def\bO{\operatorname{bO}}
\def\gbO{\operatorname{gbO}}
\newcommand{\opT}{\mathscr{T}}
\def\varx{x}
\def\vary{y}
\def\varX{X}

\def\rset{\mathbb{R}}
\def\nset{\mathbb{N}}
\def\GaussianRV{\xi}
\def\procLine{\Gamma}
\def\acceptMean{a}
\def\speedMeasure{\mathrm{h}}
\def\TV{\operatorname{TV}}
\newcommand{\defSet}[1]{\left \lbrace  #1 \right \rbrace}
\newcommand{\expeLaw}[2]{\mathbb{E} \left[ #2 \right]}
\newcommand{\bracket}[1]{\left( #1 \right)}
\newcommand{\lambdaHessianfMin}{\lambda_{\min}\left(D f(\varx)\right)}
\newcommand{\lambdaHessianfMax}{\lambda_{\max}\left(D f(\varx)\right)}
\newcommand{\nArrow}{\vec{n}}
\newcommand{\normalX}{\nArrow_{x}}

\def\Gscr{\mathscr{G}}
\def\Bscr{\mathscr{B}}

\def\varw{w}
\def\brm{\mathrm{b}}

\newcommand{\testFunction}{\psi}
\newcommand{\Wbar}{\bar{\mathsf{M}}}

\def\rmi{\mathrm{i}}

\newcommand{\LRARR}[4]{{\mbox{\raise 0.4 mm \hbox{$#1$}}} \;
  \mathop{\stackrel{\displaystyle\longrightarrow}\longleftarrow}^{#3}_{#4}
  \; {\mbox{\raise 0.4 mm\hbox{$#2$}}}}

\title{Fast Langevin based algorithm for MCMC in high dimensions
}

\author{
Alain Durmus\textsuperscript{1},
Gareth O. Roberts\textsuperscript{2}\footnote{Supported by EPSRC under grants EP/K014463/1 and EP/D002060/1.},
Gilles Vilmart\textsuperscript{3}\footnote{Partially supported by the Swiss National Science Foundation, grants 200020\_144313/1 and 200021\_162404.}, 
and Konstantinos C. Zygalakis\textsuperscript{4}\footnote{Partially supported by a grant  from the Simons Foundation and  by the Alan Turing Institute under the EPSRC grant EP/N510129/1. Part of this work was done during the author's stay at the Newton Institute for the program Stochastic Dynamical Systems  in Biology:  Numerical Methods and Applications.}
}

\begin{document}
\footnotetext[1]{LTCI, Telecom ParisTech
46 rue Barrault, 75634 Paris
Cedex 13, France.
 alain.durmus@telecom-paristech.fr }

 \footnotetext[2]{Dept of Statistics
University of Warwick,
Coventry,
CV4 7AL, UK.  Gareth.O.Roberts@warwick.ac.uk 
 }

\footnotetext[3]{Universit\'e de Gen\`eve, Section de math\'ematiques, 2-4 rue du Li\`evre, CP 64, 1211 Gen\`eve 4, Switzerland.
 Gilles.Vilmart@unige.ch }

\footnotetext[4]{School of Mathematics and Maxwell Institute of Mathematical
Sciences, University of Edinburgh, James Clerk Maxwell Building,
Peter Guthrie Tait Road, Edinburgh EH9 3FD, UK. K.Zygalakis@ed.ac.uk 
}
\maketitle

\begin{abstract}
We introduce new Gaussian proposals to improve the efficiency of the standard Hastings-Metropolis algorithm in Markov chain Monte Carlo (MCMC) methods, used for the sampling from a target distribution in large dimension $d$. The improved complexity is $\mathcal{O}(d^{1/5})$ compared to the complexity $\mathcal{O}(d^{1/3})$ of the standard approach.
We prove an asymptotic diffusion limit
theorem and show that the relative efficiency of the algorithm can be characterised by its overall
acceptance rate (with asymptotical value 0.704), independently of the target distribution.
Numerical experiments confirm our theoretical findings.

\smallskip


\end{abstract}

\section{Introduction}
Consider a probability measure $\pi$ on $\rset^d$ with density again denoted by $\pi$ with respect to the Lebesgue measure. The Langevin diffusion $\{ x_t ,\ t \geq 0 \}$ associated with $\pi$ is the  solution of the following stochastic differential equation:
\begin{equation} \label{eq:langevinA}
\rmd x_t= 
\frac12 {\Sigma \nabla \log{\pi(x_t)} \rmd t}+ \Sigma^{1/2}\rmd W_t \eqsp,
\end{equation}
where $\{W_{t}, \ t \geq 0\}$ is a standard $d$-dimensional Brownian
motion, and $\Sigma$ is a given positive definite symmetric 
matrix. Under appropriate assumptions \cite{Has80} on $\pi$, it can be
shown that the dynamic generated by \eqref{eq:langevinA} is ergodic
 with unique invariant distribution $\pi$. This is a key property
of \eqref{eq:langevinA} and taking advantage of it permits to sample
from the invariant distribution $\pi$. 
 In particular, if one could solve \eqref{eq:langevinA} 
{analytically} and then take time $t$ to infinity
then it would be possible to generate samples from  $\pi$.
 However, there exists a
limited number of cases \cite{KlP92} where such an analytical formula
exists.  A standard approach is to discretise \eqref{eq:langevinA}
using a one step integrator. 
The drawback of this approach is that it introduces
a bias, because in general $\pi$ is not invariant with respect to the Markov chain defined by the discretization, \cite{TaT90,MST10,AVZ14a}.
In addition, the discretization might fail to be ergodic \cite{RT96}, even though \eqref{eq:langevinA} is geometrically ergodic.

An alternative way of sampling from $\pi$, which does not face the bias issue introduced by discretizing \eqref{eq:langevinA}, is by using the Metropolis-Hastings algorithm \cite{Ha70}.
The idea is to construct a Markov chain $\{x_j, \ j \in \mathbb{N}\}$, where at each step $j\in\mathbb{N}$, 
given $x_{j}$, a new candidate $y_{j+1}$ is generated from a proposal density $q(x_{j}, \cdot)$. This candidate is then accepted $(x_{j+1}=y_{j+1})$ with probability $\alpha(x_{j},y_{j+1})$ given by
\begin{equation}
\label{eq:acceptance_ratio}
\alpha(x,y)=\min \left(1, \frac{\pi(y)q(y,x)}{\pi(x)q(x,y)}\right) \eqsp,
\end{equation}
 and rejected $(x_{j+1}=x_{j})$ otherwise.
The resulting Markov chain
$\sequence{x}[j][\nset]$ is reversible with respect to $\pi$ and under mild assumptions is ergodic \cite{L08,RC05}.


The simplest proposals are random walks for which $q$ is the transition kernel associated with the proposal
\begin{equation} \label{eq:RWM}
y=x+\sqrt{h}\Sigma^{1/2} \xi \eqsp,
\end{equation}
where $\xi$ is a standard Gaussian random variable  in $\IR^{d}$, and leads to the well
known Random Walk Metropolis Algorithm (RMW). 
This proposal is very simple to implement, but it suffers from (relatively) high rejection rate, due to the fact that it does not use  information about $\pi$ to construct appropriate candidate moves.

 Another family of proposals commonly used, is based on the Euler-Maruyama discretization of \eqref{eq:langevinA}, for which $q$ is the transition kernel associated with the proposal
\begin{equation} \label{eq:MALA}
y=x+(h/2) \Sigma\nabla \log{\pi(x)}+\sqrt{h}\Sigma^{1/2}\xi \eqsp,
\end{equation}
where $\xi$ is again a standard Gaussian random variable in $\IR^{d}$. This algorithm is also known as the Metropolis Adjusted Langevin Algorithm (MALA), 
and it is well-established that it has better convergence properties than the RWM algorithm in general.
This method directs the proposed moves towards areas of high probability for the distribution $\pi$, using the gradient of $\log \pi$. There is now a 
growing literature on gradient-based MCMC algorithms, as exemplified through the two papers \cite{girolami:calderhead:2011,cotter:roberts:stuart:white:2014} and the 
references therein. We also mention here function space MCMC methods \cite{cotter:roberts:stuart:white:2014}. Assuming that the target measure has a density 
w.r.t.~a Gaussian measure on a Hilbert space, these algorithms are defined in infinite dimension and avoid completely the dependence on the 
dimension $d$ faced by standard MCMC algorithms.

A natural question  is if one can improve on the behaviour of MALA by incorporating more information about the properties of $\pi$ in their proposal.
A first attempt would be to use as proposal a one-step integrator with high weak order for \eqref{eq:langevinA},
as suggested in the discussion of \cite{girolami:calderhead:2011}.
Although this turns out to not be sufficient,
we shall show that,
by slightly modifying this approach and not focusing on the weak order itself, we are able to construct a new proposal with better convergence properties than MALA.
We mention that an analogous proposal is presented independently in \cite{FaS15} in a different context to improve the
strong order of convergence of MALA.

Thus our main contribution in this paper is the introduction and theoretical analysis of the fMALA algorithm ({\em fast} MALA), and its cousins which will be introduced in Section \ref{sec:theory}. These algorithms provide for the first time, implementable gradient-based MCMC algorithms which can achieve convergence in $\bigO(d^{1/5})$ iterations, thus improving on the $\bigO(d^{1/3})$ of MALA and many related methods. These results are demonstrated as a result of high-dimensional diffusion approximation results. As well as giving these order of magnitude results for high-dimensional problems, we shall also give stochastic stability results, specifically results about the geometric ergodicity of the algorithms we introduce under appropriate regularity conditions.
 
Whilst the algorithms we describe have clear practical relevance for MCMC use, it is important to recognise the limitations of this initial study of these methodologies, and we shall note and comment on two which are particularly important. In order to obtain the diffusion limit results we give, it is necessary to make strong assumptions about the structure of the sequence of target distributions as $d$ increases. In our analysis we assume that the target distribution consists of $d$ \iid~components as in the initial studies of both high-dimensional RWM and MALA algorithms \cite{RGG97,RR98}.
Those analyses were subsequently extended (see for example \cite{roberts2001optimal}) and supported by considerable empirical evidence from applied MCMC use. We also expect that in the context of this paper, our conclusions should provide practical guidance for MCMC practitioners well beyond the cases where rigorous results can be demonstrated, and we provide an example to illustrate this in Section \ref{sec:numer}.

Secondly, our diffusion limit results depend on the initial distribution of the Markov chain being the target distribution $\pi $, clearly impractical in real MCMC contexts.  The works \cite{christensen2005scaling,jourdain2014optimal} study the case of MCMC algorithms (specifically RWM and MALA  algorithms) started away from stationarity. 
On the one hand, it turns out that MALA algorithms are less robust than RWM when starting at under-dispersed values in that scaling strategies.
Indeed, optimising mixing in stationarity can be highly suboptimal in the transient phase, often with initial moves having exponentially small acceptance probabilities (in $d$). On the other hand, a slightly more conservative strategy for MALA still achieves $\bigO(d^{1/2})$ compared to $\bigO(d)$ for RWM. It is natural to expect the story for fMALA to be at least as involved as that for MALA, and we give some empirical evidence to support this in the simulations study of Section \ref{sec:numer}.
Future work will underpin these investigations with theoretical results analogous to those of \cite{christensen2005scaling,jourdain2014optimal}.
 From a practical MCMC perspective however, it should be noted that strategies which mix MALA-transient optimal scaling with fMALA-stationary optimal scaling will perform in a robust manner, both in the transient and stationary phases. 
 Two of these effective strategies are illustrated in Section \ref{sec:numer}.


The paper is organised as follows. In Section \ref{sec:der}  we provide a
heuristic for the choice of the parameter $h$ used in the proposal as a function of the dimension $d$ of the target and present three different proposals that have better complexity scaling properties than RWM and MALA.  In Section \ref{sec:theory}, we present fMALA and its variants, and prove our main results for the introduced methods.   Section \ref{sec:convergence_erg} investigates the ergodic properties of the different proposals for a wide variety of target densities $\pi$. 
{Finally, in Section \ref{sec:numer} we  present numerical results that illustrate our theoretical findings.}

\section{Preliminaries}
\label{sec:der}

In this section we discuss some key issues regarding the convergence of MCMC algorithms. In particular, in Section \ref{ssec:cc} we discuss some issues related to the computational complexity of MCMC methods in high dimensions, while in Section \ref{ssec:fd} we present a useful heuristic for understanding the optimal scaling of a given MCMC proposal, and based on this heuristic formally derive a new proposal with desirable scaling properties.

\subsection{Computational Complexity}
\label{ssec:cc}
Here we discuss a heuristic approach for selecting the parameter $h$ in all proposals mentioned above as the dimension of the space $d$ goes to infinity. In particular, we choose  $h$ proportional to an inverse power of the dimension $d$ such that
\begin{equation} \label{eq:scaling}
h \propto d ^{-\gamma} \eqsp.
\end{equation}
This implies that the proposal $y$ is now a function of: $(i)$ the current state $x$; $(ii)$ the parameter $\gamma$ through the scaling above; and $(iii)$ the random variable $\xi$ which appears in all the considered proposals. Thus $y=y(x,\xi;\gamma)$. Ideally $\gamma$ should be as small as possible so the chain makes large steps and samples are correlated as little as possible. At the same time, the acceptance probability should not degenerate to 0 as $d \rightarrow \infty$, also to prevent high correlation amongst samples. This naturally leads to the definition of a critical exponent $\gamma_{0}$ given by
\begin{equation}
\label{eq:def_gamma_0}
\gamma_{0}=\inf_{\gamma_{c} \geq 0} \left\{ \gamma_{c}:  \liminf_{d \rightarrow \infty}  \IE \left[ \alpha(x,y) \right] >0\ , \quad \forall \gamma \in [\gamma_{c},\infty)\right\} \eqsp.
\end{equation}
The expectation here is with respect to $x$ distributed according to $\pi$ and $y$ chosen from the proposal distribution. In other words, we take 
the largest possible value for $h$, as function of $d$, constrained by asking that the average acceptance probability is bounded away from zero, 
uniformly in $d$. The time-step restriction \eqref{eq:scaling} can be interpreted as a kind of Courant-Friedrichs-Lewy restriction arising in the 
numerical time-integration of PDEs.

If $h$ is of the form \eqref{eq:scaling}, with $\gamma \geq \gamma_0$, the
  acceptance probability does not  degenerate, and  the Markov chain arising from the Metropolis-Hastings
  method can be thought of as an approximation of the Langevin SDE
  \eqref{eq:langevinA}. This Markov chain travels with time-steps $h$ on the paths of
  this SDE, and therefore requires a minimal number of steps to  reach timescales of
  $\mathcal{O}(1)$ given by
\begin{equation} \label{eq:Md}
M(d)=d^{\gamma_{0}} \eqsp.
\end{equation}
  If it
takes $\mathcal{O}(1)$ for the limiting SDE to reach stationarity,
then we obtain that $M(d)$ gives the computational complexity of the algorithm.\footnote{In this definition
  of the cost one does not take into account the cost of generating a
  proposal. This is discussed in Remark \ref{rem:comp}.}

If we now consider the case of a product measure where
\begin{equation} \label{eq:prod}
\pi(x)=\pi_d(x)=Z_d\prod_{i=1}^{d}e^{g(x_{i})} \eqsp,
\end{equation}
and $Z_d$ is the normalizing constant,
then it is well known \cite{RGG97} that for the RWM it holds $\gamma_{0}=1$, while for  MALA it holds
$\gamma_{0}=1/3$ \cite{RR98}. In the next subsection, we recall the main ideas  that
allows one to obtain these scalings (valid also for some non-product cases), and derive a
new proposal which we will call the fast Metropolis Adjusted Langevin algorithm (fMALA) and which satisfies $\gamma_{0}=1/5$ in the product case, i.e. it has a better convergence scaling.
\subsection{Formal derivation}
\label{ssec:fd}
Here we explain the main idea that is used for proving the scaling of a Gaussian\footnote{We point out that Gaussianity here is not necessary but it greatly simplifies the calculations.}  proposal in high dimensions. In particular, the proposal $y$ is now of the form
\begin{equation} \label{eq:gaussian_proposal}
y= \mu(x,h)+S(x,h)\xi \eqsp,
\end{equation}
where $\xi\sim \mathcal{N}(0,\Id)$ is a standard $d$ dimensional Gaussian random variable. Note that in the case of the RWM,
\[
\mu(x,h)=x, \quad S(x,h)=\sqrt{h}\Sigma^{1/2} \eqsp,
\]
while in the case of MALA
\[
\mu(x,h)=x+(h/2)\Sigma\nabla\log{\pi(x)},  \quad S(x,h)=\sqrt{h}\Sigma^{1/2} \eqsp.
\]
The acceptance probability   can  be written in the form
\[
\alpha(x,y)=\min \{1,\exp(R_{d}(x,y))\} \eqsp,
\]
for some function $R_{d}(x,y)$ which depends on the Gaussian proposal \eqref{eq:gaussian_proposal}. Now using the fact that $y$ is related to $x$ according to \eqref{eq:gaussian_proposal}, $R_{d}(x,x)=0$, together with appropriate smoothness properties on the function $g(x)$, one can expand $R_{d}$ in powers of $\sqrt{h}$ using a Taylor expansion
\begin{equation} \label{eq:expansion}
R_{d}(x,y)= \sum_{i=1}^{k} \sum_{j=1}^{d}h^{i/2}C_{ij}(x,\xi)+ h^{(k+1)/2}L_{k+1}(x,h^{*},\xi) \eqsp.
\end{equation}
It turns out \cite{BS08} that the scaling associated with each proposal relates directly with how many of the $C_{ij}$ terms are zero in \eqref{eq:expansion}. This simplifies if we further assume that $\Sigma=\Id$ in \eqref{eq:langevinA} and  that $\pi$ satisfies \eqref{eq:prod}, because we get for all $i \in \{1,\cdots,k \}$, $j \in \{1,\cdots,j\}$, $C_{ij}(x,\xi)=C_{i}(x_{j},\xi_j)$
and \eqref{eq:expansion} can be written as

\begin{equation} \label{eq:expansion_2}
R_{d}(x,y)= \sum_{i=1}^{k} \sum_{j=1}^{d}\frac{\sqrt{h^{i}d}}{\sqrt{d}}C_{i}(x_{j},\xi_j)+ h^{(k+1)/2}L_{k+1}(x,h^{*},\xi)
 \eqsp.
\end{equation}
We then see that if $C_{i}=0$, for $ i=1,\cdots, m$, then this implies that $\gamma_{0}=1/(m+1)$.
Indeed, this value of $\gamma_{0}$ yields $h^{m+1}d=1$ and
the leading order term in \eqref{eq:expansion} becomes
\[
\frac{1}{\sqrt{d}}\sum_{j=1}^{d}C_{m+1}(x_{j},\xi_j) \eqsp.
\]
To understand the behaviour for large $d$, we typically assume conditions
to ensure that the above term has an appropriate (weak) limit. It turns out that $m+1$ is generally an odd integer for known proposals, and the above expression is frequently approximated by a central limit theorem. The second dominant term in \eqref{eq:expansion} turns out to be $C_{2(m+1)}$, although
 to turn this into a rigorous proof one also needs to be able to control the appropriate number of higher order terms, from $m+1$ to $2(m+1)$,  as well as the remainder term in the above Taylor expansion.

\subsection{Classes of proposals with $\gamma_{0}=1/5$}
We introduce new Gaussian proposals for which
$\gamma_{0}=1/5$ in \eqref{eq:Md}.
We start by presenting the simplest method, and then give two variations of it, motivated by the desire to obtain robust and stable ergodic properties (geometric ergodicity). The underlying calculations that show $C_i=0,i=1,\ldots,m$ with $m=4$ and  $\gamma_{0}=1/5$ for these methods are contained in the supplementary materials in the form of a Mathematica file.
Recall that  $f(x)=\Sigma \nabla\log{\pi(x)}$. In the sequel, we denote by $Df$ and $ D^{2}f$  the Jacobian ($d\times d$-matrix) and the Hessian ($d\times d^2$-matrix) of  $f$ respectively.
Thus  $(Df(x))_{i,j} = \frac{\partial f_i(x)}{\partial x_j} $ and
\[
D^2 f(x) = \left[ \mathbf{H}_1(x) \quad \cdots \quad \mathbf{H}_d(x) \right]  \quad , \quad \hbox{where }  \left \lbrace \mathbf{H}_i(x)  \right \rbrace _{j,k} = \frac{\partial f_i(x)} {\partial x_k \partial x_j}.
\]
Finally for all $x\in \rset^d$, $\{ \Sigma: D^{2}f(x) \} \in \rset^d$ is defined by  for $i=1,\dots,d$:
\[
\left \lbrace \Sigma: D^{2}f(x) \right \rbrace _i = \operatorname{trace} \left( \Sigma^T \mathbf{H}_i(x) \right) \eqsp.
\]
Notice that for $\Sigma=\Id$, the above quantity reduces to the Laplacian and we have $\left \lbrace \Sigma: D^{2}f(x) \right \rbrace _i = \Delta f_i$.
\begin{remark}
\label{rem:diag}
Since by assumption $\Sigma$ is positive definite, notice that the Jacobian matrix $Df(x)$ is diagonalizable for all $x\in\mathbb{R}^d$. Indeed, 
 it is similar to the symmetric matrix $\Sigma^{-1/2}Df(x)\Sigma^{1/2}  = \Sigma^{1/2} D^2 \log \pi(x) \Sigma^{1/2}$, and we use that a symmetric matrix is always diagonalizable.
This will permit us to define analytic functionals of~$Df(x)$.
\end{remark}
\subsubsection{Fast Metropolis-Adjusted Langevin Algorithm (fMALA)}
We first give a natural proposal for which $\gamma_0 = 1/5$ based on the discussion of Section~\ref{ssec:fd}. We restrict the class of proposal defined by \eqref{eq:gaussian_proposal} by setting for all $x \in \rset^d$ and $h >0$,
\begin{equation*}
\mu(x,h) = x + h \, \mu_1(x) + h^2 \mu_2(x) \eqsp, \qquad
S(x,h) = h^{1/2} S_1(x) + h^{3/2}S_2(x) \eqsp.
\end{equation*}
By a formal calculation (see the supplementary materials), explicit expressions for the functions $\mu_1,\mu_2,S_1,S_2$ have to be imposed for the four first term $C_i(x,\xi)$, $i \in \{1,2,3,4\}$, in \eqref{eq:expansion_2} to be zero. This result implies the following definition for $\mu$ and $S$: 
\begin{subequations}
  \label{eq:fMALA}
\begin{empheq}[box=\fbox]{align}
  \label{eq:new_proposal}
\mu^{\fMALA}(x,h)  	&= x+ \frac{h}{2}f(x)-\frac{h^{2}}{24} \left(
Df(x)f(x)+ \{\Sigma: D^{2}f(x)\} \right),   \\
  \label{eq:new_proposal_2}
S^{\fMALA}(x,h)	&=\left( h^{1/2}\operatorname{I}_d +(h^{3/2}/12) Df(x) \right) \Sigma^{1/2} \eqsp.
\end{empheq}
\end{subequations}
We will refer to \eqref{eq:gaussian_proposal} when $\mu,S$
  are given by \eqref{eq:fMALA} as the fast
  Unadjusted Langevin Algorithm (fULA) when viewed as a numerical method for
  \eqref{eq:langevinA} and as the fast Metropolis-Adjusted Langevin Algorithm (fMALA) when used as a proposal in the Metropolis-Hastings framework.

{
\begin{remark}
It is interesting to note that compared with Unadjusted Langevin Algorithm (ULA), fULA has the same order of weak 
convergence one, if applied as a one-step integrator for \eqref{eq:langevinA}. One could obtain a second order weak method by changing the 
constants in front of the higher order coefficients, but in fact the corresponding method would not have better scaling properties than MALA when 
used in  the Metropolis-Hastings framework. This observation answers negatively in part one of the questions in the discussion of 
\cite{girolami:calderhead:2011} about the potential use   of higher order integrators for the Langevin equation within the Metropolis-Hastings 
framework.
\end{remark}
}

\begin{remark} \label{rem:comp}
The proposal given by equation \eqref{eq:fMALA} contains higher order derivatives of the vector field $f(x)$, resulting in higher computational cost 
than the standard MALA proposal. This additional cost might offset the benefits of the improved scaling, since the corresponding Jacobian and 
Hessian can be full matrices in general.  However, there exist cases of interest\footnote{We study one of those in Section \ref{sec:numer}.} where 
due to the structure of the Jacobian and Hessian the computational cost of the fMALA proposal is of the same order with respect to the dimension 
$d$ as for the MALA proposal.  Furthermore, we note that one possible way to avoid derivatives is
by using finite differences or Runge-Kutta type approximations of the 
proposal \eqref{eq:fMALA}. This, however, is out of the scope of the present paper.
\end{remark}


\subsubsection{Modified Ozaki-Metropolis algorithm (mOMA)}
{One of the problems related to the MALA proposal is that it fails to be geometrically ergodic for a wide range of targets $\pi$ \cite{RT96}. This issue was addressed in \cite{roberts:stramer:2002} where a modification of MALA based on the Ozaki discretization \cite{ozaki:1992} of \eqref{eq:langevinA} was proposed and studied.}
In the same spirit as in \cite{roberts:stramer:2002} we propose here a modification of fMALA, defined by
\begin{subequations}
\label{eq:mO}
  \begin{empheq}[box=\fbox]{align}
\mu^{\mO}(x,h) &= \varx + \opT_1(D f (\varx) , h , 1)  f(\varx) -(h^{2}/6)D f (\varx)
 f(\varx)  \label{eq:mOa}\\
 \nonumber
 &\phantom{\varx + \opT_1(D f (\varx) , h , 1)  f(\varx) -(h^{2}}-(h^{2}/24) \{\Sigma: D^{2}f(x)\}  \\
S^{\mO}(x,h) & = \bracket{\opT_1(D f(\varx) , 2 h , 1 )  - (h^{2}/3) D f (\varx)}^{1/2} \Sigma^{1/2}   \eqsp.
\end{empheq}
\end{subequations}
where
\begin{equation}
\label{eq:operatorOne}
\opT_1(\Mrm,h,a) =
(a \Mrm)^{-1}(\rme^{(ah/2) \Mrm} - \Id)
\end{equation}
for all\footnote{%
Notice that the matrix functionals in \eqref{eq:operatorOne},\eqref{eq:operatorTwo},\eqref{eq:operatorThree} remain valid if matrix $aM$ is not invertible,
using the appropriate power series for the matrix exponentials.
}
$\Mrm \in \IR^{d \times d},\, \ h>0$, $a \in \IR$. 

The Markov chain defined by \eqref{eq:mO} will be referred to as the
modified unadjusted Ozaki algorithm (mUOA), whereas when it is used in
a Hastings-Metropolis algorithm, it will be referred to as the modified
Ozaki Metropolis algorithm (mOMA). 
Note that $t \mapsto (\rme^{h t } - 1) / t - (1/3)h^2 t$ is positive
on $\IR$ for all $h >0$. It then follows from
  Remark~\ref{rem:diag} that for all $x \in \rset^d$, the matrix
  $\opT_1(D f(\varx) , 2 h , 1 ) - (h^{2}/3) D f (\varx)$ is
  diagonalizable with non-negative eigenvalues, which permits to define its matrix square-root, and $S^{\mO}(x,h)$ is
  well defined for all $\varx \in \IR^d$ and $h>0$.
\begin{remark} \label{rem:mOMA_rem}
In regions where $\norm{\Sigma \nabla \log \pi(x)}$ is much greater than  $\norm{x}$, we need in practice to take $h$ very small (of order $\norm{x} / \norm{\Sigma \nabla \log \pi (x)}$)
for MALA 
to exit these regions. However such a choice of $h$
depends on $x$ and cannot be used directly. Such a value of $h$ can therefore be hard to find theoretically as well as computationally. This issue can be tackled by multiplying $f=\Sigma\nabla \log \pi(x)$ by  $\opT_1(Df(x),h,a)$ in \eqref{eq:mOa}. Indeed under some mild conditions, in that case, we can obtain an algorithm with good mixing properties 
for all $h>0$ ; see \cite[Theorem 4.1]{roberts:stramer:2002}. mOMA faces similar problems due to the term $Df(x) f(x)$.
\end{remark}


\subsubsection{Generalised Boosted Ozaki-Metropolis Algorithm (gbOMA)}
{Having discussed the possible limitations of mOMA  in Remark \ref{rem:mOMA_rem} we generalise here the approach in \cite{roberts:stramer:2002} to deal with the complexities arising to the presence of the  $Df(x)f(x)$ term}. In particular we now define
\begin{subequations}
  \label{eq:gbO}
  \begin{empheq}[box=\fbox]{align}
\mu^{\gbO}(x,h)  &= \varx + \opT_1(D f (\varx), h,a_1)  f(\varx)
\nonumber\\
&\qquad - (1/3)\opT_3(D f(\varx),h,a_3) \{ \Sigma: D^{2}f(x)\}  \nonumber \\ 
&\qquad +
\left((a_1/2) + (1/6)\right) \opT_2(D f(\varx) , h , a_2)   f(\varx)
\eqsp, \\
\nonumber
 S^{\gbO}(x,h) & = \left( \opT_1 ( D f(\varx), 2 h , a_4) \right.\\
&\qquad \left. + \left((a_4/2) - (1/6)\right) \opT_2(D
 f(\varx), 2 h , a_5)\right) ^{1/2} \Sigma^{1/2}  \eqsp.
\label{eq:gbOb}
\end{empheq}
\end{subequations}

where $a_i,\  i=1, \cdots, 5$ are positive parameters, $\opT_1$ is given by
\eqref{eq:operatorOne} and
\begin{align}
\label{eq:operatorTwo}
\opT_2(\Mrm,h,a) &=
(a \Mrm)^{-1}(\rme^{- (a h^2 /4) \Mrm^2} - \Id)
\\
\label{eq:operatorThree}
\opT_3(\Mrm,h,a) &=
(a \Mrm)^{-2}(\rme^{(ah/2) \Mrm} - \Id - (ah/2) \Mrm)
\end{align}
with $\Mrm \in \IR^{d \times d}$, $h>0$, $a \in \IR$ and $\operatorname{I}_d$ is the identity matrix. The Markov chain defined by \eqref{eq:gbO} will be referred to as
the generalised boosted unadjusted Ozaki algorithm (gbUOA), whereas when it is used in a
Hastings-Metropolis algorithm, it will be referred to as the generalised
boosted Ozaki Metropolis algorithm (gbOMA). Note that $S^{\gbO}$ in \eqref{eq:gbOb} is not always well defined in
general. However, using Remark~\ref{rem:diag}, the following condition is sufficient to define $S^{\gbO}$ with the square-root of a diagonalizable matrix with non-negative eigenvalues.
\begin{assumption}
\label{hyp:definite_positive_covariance_matrix}
The function $t \mapsto (\rme^{a_4 t
} - 1)/(a_4 t) + (a_4/2 - (1/6)) (\rme^{-a_5  t^2}-1)/(a_5 t)$ is positive on $\IR$.
\end{assumption}
For $a_4 = a_5 =1$, 
 this assumption is satisfied, and
choosing $a_i =1$ for all $i = 1 , \dots, 5$, \eqref{eq:gbO} leads to a well defined
proposal, which will be referred to as  the boosted Unadjusted Ozaki Algorithm (bUOA), whereas
when it is used in a Hastings-Metropolis algorithm, it will be referred to as the boosted
Ozaki Metropolis Algorithm (bOMA). We will see in Section~\ref{sec:convergence_erg} that
bOMA has  nicer ergodic properties than fMALA.

\section{Main scaling results}
\label{sec:theory}

In this section, we  present  the optimal scaling results for  fMALA and gbOMA introduced in Section~\ref{sec:der}. We recall from the discussion in Section \ref{sec:der}
 that the parameter $h$  depends on the dimension and is given as $h_d = \ell^2 d^{-1/5}$, with $\ell >0$. Finally, we prove our results for the case of target distributions
of the product form given by (\ref{eq:prod}), we take $\Sigma= \Id$, and make the following assumptions on $g$.
\begin{assumption} We assume
  \label{hyp:log_density}
  \begin{enumerate}
  \item \label{item:hyp_log_density_lipschitz_Cten}   $g \in C^{10}(\IR)$ and $g''$ is bounded on $\IR$.
    \item  \label{item:hyp_log_density_polynomial}
		The derivatives of $g$ up to order $10$ have at most a polynomial growth, i.e. there exists constants $C,\kappa$ such that
  \[ |g^{(i)}(t)| \leq C(1+|t|^\kappa),\qquad t\in \IR, i=1,\ldots,10.
  \]
\item \label{item:hyp_log_density_moments} for all $k \in \IN$,
    \[ \int_\IR t^k \rme^{g(t)} \rmd t < \plusinfty \eqsp.
    \]
\end{enumerate}
\end{assumption}

\subsection{Optimal scaling of fMALA}
\label{ssec:optimal_scaling_fMALA}

The Markov chain produced by fMALA, with target density $\pi_d$ and started at stationarity, will be denoted by $\{ \varX_k^{d,\fMALA}\ ,$ $  \ k \in \IN \}$.
Let $q^{\fMALA}_d$ be the transition density associated with the proposal of fMALA relatively to $\pi_d$. In a similar manner, we denote by $\alpha_d ^{\fMALA}$ the acceptance probability.
Now we introduce the jump process based on $\{\varX_k^{d,\fMALA}, \ k \in \IN\}$, which
allows us to compare this Markov chain to a continuous-time process. Let $\{J_t , \  t
  \in \IR_+\}$ be a Poisson process with rate $d^{1/5}$, and let $\procLine^{d,\fMALA} =
\{\procLine^{d,\fMALA}_t, \  t \in \IR_+ \}$ be the $d$-dimensional jump process defined by
$\procLine^{d,\fMALA}_t = \varX^{d,\fMALA}_{J_t}$. We denote by
 \[
 \acceptMean_d^{\fMALA}(\ell) = \int_{\rset^d} \int_{\rset^d} \pi_d(\varx) q_d ^{\fMALA}(\varx, \vary) \alpha_d ^{\fMALA}(\varx,
 \vary) \rmd \varx \rmd \vary
 \]
the mean under $\pi_d$ of the acceptance rate.

\begin{theorem}
\label{theo:accept_ratio_fMALA}
Assume Assumption~\ref{hyp:log_density}.
Then
\[
\lim_{d \to \plusinfty } \acceptMean_d^{\fMALA}(\ell) = \acceptMean^{\fMALA} (\ell)\eqsp,
\]
  where $\acceptMean^{\fMALA}(\ell) = 2 \Phi(-K^{\fMALA} \ell ^5 /2)$ with $\Phi(t) = (1/(2 \uppi))
  \int_{- \infty} ^t \rme^{-s^2/2} \rmd s$ and the expression of $K^{\fMALA}$ is
  given in Appendix~\ref{app:expression_K}.
  \end{theorem}

\begin{theorem}
\label{theo:scaling:fMALA}
Assume Assumption~\ref{hyp:log_density}.
Let $\{Y_t^{d,\fMALA}  = \procLine^{d,\fMALA}_{t,1}  , \  t \in \IR_+ \} $ be the process corresponding to
the first  component of $\procLine^{d,\fMALA}$. Then, $\{Y^{d,\fMALA}, \ d \in \IN^*\}$ converges weakly (in the Skorokhod
topology), as $d \to \infty$, to the solution $\{Y_t^{\fMALA} , \ t \in \IR_+\}$ of the Langevin equation defined
by:
\begin{equation}
\label{eq:langevin_fMALA}
\rmd Y_t ^{\fMALA}  = (\speedMeasure^{\fMALA}(\ell))^{(1/2)} \rmd B_t + (1/2) \speedMeasure^{\fMALA}(\ell) \nabla \log \pi_1(Y_t^{\fMALA}) \rmd t
\eqsp,
\end{equation}
where $\speedMeasure^{\fMALA} (\ell) = 2 \ell^2 \Phi(-K^{\fMALA} \ell^5/2)$ is the speed of the
limiting diffusion. Furthermore, $\speedMeasure ^{\fMALA} (\ell)$ is maximised at the unique value of $\ell$
for which $\acceptMean^{\fMALA} (\ell) = 0.704343$.
\end{theorem}

\begin{proof}
The proof of these two theorems are in Appendix \ref{app:scaling_proof}.
\end{proof}

\begin{remark}
The above analysis shows that for
  fMALA, the optimal exponent defined in 
  \eqref{eq:def_gamma_0} is given by $\gamma_0=1/5$ as discussed in Section \ref{ssec:fd}.
Indeed, if $h_d$ has the form $\ell^2 d^{-1/5 + \epsilon}$, then an
   adaptation of the proof of Theorem~\ref{theo:accept_ratio_fMALA} implies
   that for all $\ell >0$, if $\epsilon \in \ooint{0,1/5}$, 
$\lim_{d
    \to + \infty} \acceptMean^{\fMALA} (\ell) = 0$. 
 In contrast, if $\epsilon <0$ then $\lim_{d \to + \infty}
  \acceptMean^{\fMALA} (\ell) = 1$. 
\end{remark}

\subsection{Scaling results for gbOMA}
As in the case of fMALA, we assume $\pi_d$ is of the form
\eqref{eq:prod} and we take $\Sigma = \Id$, $h_d = \ell^2
d^{-1/5}$.  The Metropolis-adjusted Markov chain based on gbOMA, with
target density $\pi_d$ and started at stationarity, is denoted by
$\{\varX_k^{d,\gbO}, \ k \in \IN \}$.  We will denote by $q^{\gbO}_d$
the transition density associated with the proposals defined by gbOMA
with respect to $\pi_d$. In a similar manner, the acceptance
probability relatively to $\pi_d$ and gbOMA will be denoted by
$\alpha_d ^{\gbO}$.  Let $\{J_t ,\ t \in \IR_+ \}$ be a Poisson
process with rate {$d^{1/5}$}, and let $\procLine^{d,\gbO} =
\{\procLine^{d,\gbO}_t, \ t \in \IR_+ \}$ be the $d$-dimensional jump
process defined by $\procLine^{d,\gbO}_t =
\varX^{d,\gbO}_{J_t}$. Denote also by
 \[
 \acceptMean_d^{\gbO}(\ell) = \int_{\rset^d} \int_{\rset^d} \pi_d(\varx) q_d ^{\gbO}(\varx, \vary) \alpha_d ^{\gbO}(\varx,
 \vary) \rmd \varx \rmd \vary
 \]
the mean under $\pi_d$ of the acceptance rate of the algorithm.

\begin{theorem} \label{theo:accept_ratio_gbO}
Assume Assumptions~\ref{hyp:definite_positive_covariance_matrix} and
\ref{hyp:log_density}. Then
\[
\lim_{d \to \plusinfty } \acceptMean_d^{\gbO}(\ell) = \acceptMean^{\gbO}(\ell) \eqsp,
\]
  where $\acceptMean^{\gbO} (\ell)= 2 \Phi(-K^{\gbO} \ell ^5 /2)$ with $\Phi(t) = (1/(2 \uppi))
  \int_{- \infty} ^t \rme^{-s^2/2} \rmd s$ and  $K^{\gbO}$ are
  given in Appendix~\ref{app:expression_K}.
  \end{theorem}

\begin{theorem} \label{theo:scaling:gbO}
Assume Assumptions~\ref{hyp:definite_positive_covariance_matrix} and
\ref{hyp:log_density}.
Let $\{G_t^{d,\gbO}  = \procLine^{d,\gbO}_{t,1} ,\   t \in \IR_+\} $ be the process corresponding to
the first  component of $\procLine^{d,\gbO}$. Then, $\defSet{G^{d,\gbO},\ d \in \IN^*}$ converges weakly (in the Skorokhod
topology) to the solution $\{G_t^{\gbO} ,\  t \in \IR_+ \}$ of the Langevin equation defined
by:
\[
\rmd G_t ^{\gbO}  = (\speedMeasure^{\gbO}(\ell))^{(1/2)} \rmd B_t + (1/2) \speedMeasure^{\gbO}(\ell) \nabla \log \pi_c(G_t^{\gbO}) \rmd t
\eqsp,
\]
where $\speedMeasure^{\gbO} (\ell) = 2 \ell^2 \Phi(-K^{\gbO} \ell^5/2)$ is the speed of the
limiting diffusion. Furthermore, $\speedMeasure ^{\gbO} (\ell)$ is maximised at the unique value of $\ell$
for which $\acceptMean^{\gbO} (\ell) = 0.704343$.
\end{theorem}

\begin{proof}

Note that under Assumption~\ref{hyp:log_density}-\ref{item:hyp_log_density_lipschitz_Cten}, at fixed $a>0$, using the regularity properties of
$(x,h) \mapsto \opT_i(x,h,a)$ on $\rset^2$ for $i = 1, \dots,3$, there exists an open interval $I$, which
contains $0$, and $M_0 \geq 0$ such that for all $x \in \IR$, $k=1, \cdots, 11$, and $i=1, \cdots, 3$
\[
\abs{\frac{\partial^k\left( \opT_i(g''(x), h , a )\right)} {\partial h ^k}} \leq M_0 \ \ \forall h \in I \eqsp.
\]
Using in addition Assumption~\ref{hyp:definite_positive_covariance_matrix} there exists $m_0 >0$
such that for all $h \in I$ and for all $\varx \in \IR$,
\[
 \opT_1 ( g''(\varx), 2 h , a_4)  + \left((a_4/2) - (1/6)\right) \opT_2(g''(\varx), 2 h , a_5) \geq m_0 \eqsp.
\]
Using these two results, the proof of both theorems follows the same lines as
Theorems~\ref{theo:accept_ratio_fMALA} and \ref{theo:scaling:fMALA}, which can be found in
Appendix~\ref{app:scaling_proof}.
  \end{proof}

\section{Geometric ergodicity results for high order Langevin schemes}
\label{sec:convergence_erg}
Having established the scaling behaviour of the different proposals in the previous section,
  we now proceed with establishing geometric ergodicity results for our new Metropolis algorithms.  Furthermore, for completeness, we study the
behaviour of the corresponding unadjusted proposal. For simplicity, we
will take in the following $\Sigma = I_d$ and we limit our study of gbOMA to the one of
bOMA, which is given by:
\begin{align}
  \label{eq:bO}
  y^{\bO} &= \mu^{\bO}(x,h)  + S^{\bO}(x,h) \ \xi \eqsp,  \\
  \nonumber \mu^{\bO}(x,h)&= \varx + \opT_1(D f (\varx) , h , 1)
  f(\varx) + (2/3) \opT_2(D f
  (\varx) , h , 1)   f(\varx) \\
  \nonumber &- (1/3)
  \opT_3(D f(\varx) , h, 1)  \{\Sigma: D^{2}f(x)\} \eqsp, \\
  \nonumber S^{\bO}(x,h)&= \bracket{ \opT_1(D f (\varx) , 2 h , 1) + (1/3)
  \opT_2(D f (\varx) , 2 h , 1)}^{1/2} \eqsp,
\end{align}
where $\opT_1$, $\opT_2$ and $\opT_3$ are respectively defined by \eqref{eq:operatorOne}, \eqref{eq:operatorTwo} and \eqref{eq:operatorThree}.
First, let us begin with some definitions. For a signed measure $\nu$ on $\rset^d$, we
define the total variation norm of $\nu$ by
\[
\norm{\nu}_{\TV} = \sup_{A \in \mathcal{B}(\rset^d)} \abs{\nu(A)} \eqsp,
\]
where $\mathcal{B}(\rset^d)$ is the Borel $\sigma$-algebra of $\rset^d$.
Let $P$ be a Markov kernel with invariant measure $\pi$. For a given measurable function $V: \rset^d \to \coint{1,\plusinfty}$, we
will say that $P$ is $V$-geometrically ergodic if there exist $C \geq 0$ and $\rho \in
\coint{0,1}$ such that for all $x \in \rset^d$ and $n \geq 0$
\[
\norm{P^n(x, \cdot) - \pi }_V \leq C \rho^n V(x) \eqsp,
\]
where for $\nu$ a signed measure on $\rset^d$, the $V$-norm $\norm{\cdot }_V$ is defined by
\[
\norm{\nu} = \sup_{\{ f\ ; \ \abs{f} \leq V \}} \int_{\rset^d} f(x) \nu(\rmd x ) \eqsp.
\]
We refer the reader to \cite{meyn:tweedie:2009} for the definitions of small sets,
$\varphi$-irreducibility and transience. Let $P$ be a Markov kernel on $\rset^d$,
$\Leb^d$-irreducible, where $\Leb^d$ is the Lebesgue measure on $\IR^d$, and aperiodic and $V :
\rset^d \to \coint{1,\plusinfty}$ be a measurable function. In order to establish that
$P$ is $V$-geometric ergodicity, a sufficient and necessary condition is given by
a geometrical drift (see \cite[Theorem 15.0.1]{meyn:tweedie:2009}), namely  for some small set
$\smallSet$, there exist $\lambda <1$ and $b < \plusinfty$ such that for all $x \in \rset^d$:
\begin{equation}
\label{eq:drift_condition}
PV(x) \leq \lambda V(x) + b \1_{\smallSet}(x) \eqsp.
\end{equation}

Note that the different
considered proposals belong to the class of Gaussian Markov
kernels. Namely, let $Q$ be a Markov kernel on $\rset^d$. We say that
$Q$ is a Gaussian Markov kernel if for all $x \in \rset^d$, $Q(x,\cdot)$
is a Gaussian measure, with mean $\mu(x)$ and covariance matrix
$
{S(x)S^{T}(x)}$, where $x \mapsto \mu(x)$ and $x \mapsto S(x)$ are measurable
functions from $\rset^d$ to respectively $\rset^d$ and $\mathcal{S}^*_+(\rset^d)$, the set of symmetric {positive definite matrices}
 of dimension $d$. These two
functions will be  referred to as the mean value map and the the
variance map {respectively}. The Markov kernel $Q$ has transition density $q$ given by:
{\begin{equation}
\label{eq:transition_density_Gaussian}
q(\varx,\vary) = \frac{1}{{(2 \uppi)^{d/2} \abs{S(\varx)}}}
\exp\left( -(1/2)\ps{S(\varx)^{-2} (\vary-\mu(\varx))}{(\vary-\mu(\varx))}\right) \eqsp,
\end{equation}}
where for $\Mrm \in \rset^{d \times d} $, $\abs{\Mrm}$ denotes the determinant of $\Mrm$.
Geometric ergodicity of Markov Chains with Gaussian Markov kernels and the corresponding
  Metropolis-Hastings algorithms was the subject of study of \cite{RT96,hansen:2003}. But
  contrary to \cite{hansen:2003}, we assume for simplicity the following assumption on the
  functions $\mu:\R^d\rightarrow \R^d$ and $S:\R^d\rightarrow \mathcal{S}^*_+(\rset^d)$:
\begin{assumption}
\label{hyp:Gaussian_Markov_kernel}
The functions $x \mapsto \mu(x)$ and $ x \mapsto S(x)$ are continuous. 
\end{assumption}
Note that if $\pi$, a target probability
measure on $\IR^d$, is absolutely continuous with respect to the
Lebesgue measure with density still denoted by $\pi$, the following
assumption ensures that the various different proposals introduced in this paper satisfy Assumption
\ref{hyp:Gaussian_Markov_kernel}:
\begin{assumption}
\label{hyp:density_convergence}
 The log-density $g$ of $\pi$ belongs to  $C^3(\IR^d)$.
\end{assumption}
We proceed in Section \ref{subsec:convergence_adjusted} with presenting and extending
where necessary the main results about geometric ergodicity of Metropolis-Hasting algorithms using
Gaussian proposals. In Section \ref{subsec:presentation_potential}, we then introduce two
different potential classes on which we apply our result in Section
\ref{subsec:application_adjusted}. Finally in
Section~\ref{subsec:convergence_unadjusted}, for completeness, we make the same kind of
study but for unadjusted Gaussian Markov kernels on $\rset$.



\subsection{Geometric ergodicity of Hastings-Metropolis algorithm based on Gaussian Markov kernel}
\label{subsec:convergence_adjusted}

We first present an extension of the result given in in \cite{hansen:2003} for geometric ergodicity of
 Metropolis-Hastings algorithms based on Gaussian proposal kernels.
 In particular,  let $Q$ be a Gaussian Markov kernel with mean value map and variance map satisfying Assumption \ref{hyp:Gaussian_Markov_kernel}.
We use such proposal in a Metropolis
algorithm with target density $\pi$ satisfying Assumption \ref{hyp:density_convergence}.
Then, the produced Markov kernel $P$ is given by
\begin{equation}
  \label{eq:kernel_MH}
P(\varx , \rmd \vary) = \alpha(\varx,\vary) q(\varx,\vary) \rmd \vary + \delta_{\varx}(\rmd \vary)
\int_{\IR^d}(1- \alpha(\varx,\vary)) q(\varx,\vary) \rmd \vary \eqsp,
\end{equation}
where $q$ and $\alpha$ are resp. given by \eqref{eq:transition_density_Gaussian} and  \eqref{eq:acceptance_ratio}.

\begin{assumption}
\label{hyp:ratio_convergence_bO}
We assume
$\liminf_{\norm{\varx} \to \plusinfty} \int_{\IR^d} \alpha(\varx,\vary) q(\varx,\vary) \rmd
\vary > 0 $.
  \end{assumption}
Note that this condition is necessary to obtain the geometric ergodicity of a Metropolis-Hastings algorithm by \cite[Theorem 5.1]{roberts:tweedie:1996}.
We shall follow a well-known technique in MCMC theory in demonstrating that Assumption~\ref{hyp:ratio_convergence_bO} allows us to ensure that
geometric ergodicity of the algorithm is inherited from that of the proposal Markov chain itself.
Thus, in the following lemma we combine the conditions given by \cite{hansen:2003}, which imply geometric ergodicity of Gaussian Markov kernels, with Assumption ~\ref{hyp:ratio_convergence_bO} to get geometric ergodicity of the resultant Metropolis-Hastings Markov kernels.

\begin{lemma}
  \label{propo:convergence_gaussian_proposal}
  Assume Assumptions
  \ref{hyp:Gaussian_Markov_kernel}, \ref{hyp:ratio_convergence_bO}, and
  there exists $\tau \in \ooint{0,1}$ such that
\begin{equation}
 \label{eq:hyp:gaussian_prop_mu}
\limsup_{\norm{\varx} \to \plusinfty} \norm{\mu(\varx)}/ \norm{\varx} = \tau, \text{ and }  \limsup_{\norm{\varx} \to \plusinfty} \norm{S(\varx)}/ \norm{\varx} = 0 \eqsp.
\end{equation}
Then, the Markov kernel $P$ given by \eqref{eq:kernel_MH} are
$V$-geometrically ergodic, where $V(\varx) = 1+\norm{\varx}^2$.
  \end{lemma}

\begin{proof}
The proof is postponed to Appendix~\ref{subsec:lemme_convergence_gaussian_proposal}.
\end{proof}
We now provide some conditions which imply that $P$ is not geometrically ergodic.
\begin{theorem}
\label{theo:non_geo_erg_Gaussian_MH}
Assume Assumptions \ref{hyp:Gaussian_Markov_kernel},\ref{hyp:density_convergence}, that $\pi$ is bounded and there exists $\epsilon >0$ such that
\begin{equation}
  \label{eq:assum_theo_non_geo}
\liminf_{\norm{x} \to \plusinfty} \norm{S(x)^{-1} \mu(x)}\norm{x}^{-1} > \epsilon^{-1} \eqsp, \qquad  \liminf_{\norm{x} \to \plusinfty}  \inf_{\norm{y} = 1} \norm{S(x) y } \geq \epsilon  \eqsp,
\end{equation}
and
\begin{equation}  \label{eq:assum_theo_non_geo2}
\lim_{\norm{x} \to \plusinfty}\log\left(\abs{S(x)} \right) / \norm{x}^2 = 0 \eqsp.
\end{equation}
Then, $P$ is not geometrically ergodic.
\end{theorem}
\begin{proof}
The proof is postponed to Appendix~\ref{subsec:proof_theo_non_geo}.
  \end{proof}

\subsection{Exponential potentials}
\label{subsec:presentation_potential}
We illustrate our results on the following classes of density.

\subsubsection{The one-dimensional class $\Escr(\beta,\gamma)$}
Let $\pi$ be a probability density on $\IR$ with respect to the Lebesgue measure. We will say that $\pi \in
\Ecal(\beta,\gamma)$ if $\pi$ is positive, belongs to $C^3(\rset)$ and there exist $R_\pi,\beta >0$ such that
for all $x \in \rset$, $\abs{x} \geq R_\pi$,
\[
\pi(x) \propto \rme^{-\gamma \abs{x}^\beta} \eqsp.
\]
Then for $\abs{x} \geq R_\pi$, $\log(\pi(x))' = -\gamma \beta x \abs{x}^{\beta-2}$, $\log(\pi(x))'' = -\gamma  \beta (\beta-1) \abs{x}^\beta/x^2$ and
$ \log(\pi(x))^{(3)} = -\gamma  \beta (\beta-1)(\beta-2)\abs{x}^\beta/x^3$.
\subsubsection{The multidimensional exponential class $\Pscr _m$}
\label{sec:poly_homo}
Let $\pi$ be a probability density on $\IR^d$ with respect to the Lebesgue measure. We
will say that $\pi \in \Pscr_m$ if it is positive, belongs to $C^3(\rset^d)$ and there exists $R_\pi \geq 0$ such that for all $\varx \in \rset^d$,
$\norm{\varx} \geq R_\pi$,
\[
\pi(\varx) \propto \rme^{-\qrm (\varx)} \eqsp,
\]
where $\qrm$ is a function of the following form. There exists a homogeneous polynomial $\prm$
of degree $m$ and a three-times continuously differentiable function $\rrm$ on $\IR^d$
satisfying
\begin{equation}
  \label{eq:petit_to_r_poly_homog}
\norm{D^2(\nabla r)(\varx)} \underset{\norm{x} \to \plusinfty}{=} o (\norm{\varx}^{m-3}) \eqsp,
\end{equation}
and for all $x \in \rset^d$
\[
\qrm ( \varx ) = \prm (\varx) + \rrm (\varx) \eqsp.
\]
Recall that $\prm$ is an homogeneous polynomial of degree $m$ if for all $t \in \IR$ and $\varx
\in \IR^d$, $\prm(t \varx) = t^m \prm (\varx)$. Finally we define $\Pscr_m^+$,
the set of density $\pi \in \Pscr_m$ such that the Hessian  of $\prm$ at $x$, $\nabla^2 \prm(x)$ is positive definite for
all $x \not = 0$. \\
When $\prm$ is an homogeneous polynomial of degree $m$, it can be written as
\[
\prm(\varx) = \sum_{\abs{\kbf} = m} a_{\kbf} \varx^{\kbf} \eqsp,
\]
where $\kbf \in \IN^d$, $\abs{\kbf} = \sum_i k_i$ and $\varx^{\kbf} = x_1^{k_1} \cdots
x_d^{k_d}$. Then denoting by $\normalX = \varx/\norm{\varx}$, it is easy to see that the following relations holds for all $\varx \in \rset^d$.
\begin{align}
  \label{eq:orig_poly_homog}
\prm(\varx)& = \norm{x}^m \ \prm(\normalX) \\
  \label{eq:grad_poly_homog}
\nabla \prm(\varx)& = \norm{x}^{m-1} \nabla \prm(\normalX) \\
\label{eq:hessian_poly_homog}
\nabla^2 \prm(\varx) &= \norm{x}^{m-2} \nabla^2 \prm(\normalX) \\
\label{eq:third_derivatives_poly_homog}
D^2(\nabla \prm)(x)  &= \norm{x}^{m-3} D^2 (\nabla \prm)(x) \\
\label{eq:grad_ps_poly_homog}
\ps{\nabla \prm(\varx)}{\varx} &= m\ \prm(\varx) \\
  \label{eq:hessian_ps1_poly_homog}
\nabla^2 \prm(\varx) \varx &= (m-1)\ \nabla \prm (\varx) \\
  \label{eq:hessian_ps1_poly_homog2}
\ps{\nabla^2 \prm (\varx) \varx}{\varx} &= m (m-1)\  \prm(\varx)\eqsp.
\end{align}
From \eqref{eq:hessian_poly_homog}, it follows that $\nabla^2 \prm (\varx)$ is definite
positive for all $\varx \in \IR^d \setminus {0}$ if and only if $\nabla^2 \prm (\nArrow)$ is
positive definite for all $\nArrow$, with $\norm{\nArrow}=1$. Then, $\prm$ belongs to $\Pscr_m^+$ only if
$m \geq 2$.

\subsection{Geometric ergodicity of the proposals: the case of Metropolis-Hastings algorithms}
\label{subsec:application_adjusted}
{In this section we study the behaviour of our proposals within the Metropolis-Hastings framework. We will split our investigations in two parts: in the first we study fMALA and mOMA; while in the second we have a more detailed look in the properties of bOMA not only for the class  $\Escr(\beta,\gamma)$, but also for the polynomial class $\Pscr_m^+$.}

\subsubsection{Geometric ergodicity of fMALA, mOMA  for the class $\Escr(\beta,\gamma)$}
In the case $\beta \in \ooint{0,2}$, fMALA and mOMA have their mean map behaving like
$x -\beta \gamma x \abs{x}^{\beta-2}/2$ at infinity and their variance map bounded from above. This is exactly the behaviour that MALA \cite{RT96} has for the same values of $\beta$, thus one would expect them to behave in the same way.  This is indeed the case and thus using the same reasoning as in the proof  \cite[Theorem 4.3]{RT96} we deduce that the two algorithms are not geometrically ergodic for $\beta \in \ooint{0,1}$. Similarly, the proof in \cite[Theorem 4.1]{RT96} can be used to show that the two algorithms are geometrically ergodic for $\beta \in \coint{1,2}$. Furthermore, for values of $\beta \geq 2$ we have the following cases

\begin{enumerate}[(a)]
\item For $\beta =2$,
\begin{enumerate}[-]
\item  fMALA is geometrically ergodic if $h\gamma(1+h
  \gamma /6)\in \ooint{0,2}$ by  \cite[Theorem 4.1]{RT96},
  and not geometrically ergodic if $ h\gamma(1+h
  \gamma /6) >2$ by Theorem \ref{theo:non_geo_erg_Gaussian_MH},
  since $\mu^{\fMALA}$ is equivalent at infinity to $(1-h\gamma(1+h
  \gamma /6))x$ and $S^{\fMALA}(x)$ is constant for $\abs{x} \geq
  R_\pi$.
\item
Since $\mu^{\mO}$ is equivalent at
infinity to $(\rme^{-\gamma h}-2(h\gamma)^2/3 )x$, we observe that  mOMA is geometrically
ergodic if $h \gamma \in
  \ooint{0,1.22}$ by \cite[Theorem 4.1]{RT96},
  and   not geometrically ergodic if $h \gamma >1.23$ by Theorem \cite[Theorem 5.1]{roberts:tweedie:1996}.
\end{enumerate}
\item
For $\beta >2$,
 fMALA and mOMA are  not geometrically ergodic  by Theorem
 \ref{theo:non_geo_erg_Gaussian_MH} since the mean value maps of their proposal kernels
 are equivalent at infinity to $-C_1 \abs{x}^{2 \beta -2}/x$, their variance map to
 $C_2 \abs{x}^{\beta-2}$ for some constants $C_1,C_2 >0$, and the variance maps are bounded from below.
\end{enumerate}

\subsubsection{Geometric ergodicity of bOMA}
In this section, we give some conditions under which bOMA is
geometrically ergodic and some examples of density which satisfy such
conditions.  For a matrix $\Mrm \in \IR^{d \times d}$, we denote $\lambdaMin {\Mrm} =
\min \Sp(\Mrm)$ and $\lambdaMax {\Mrm} = \max \Sp(\Mrm)$, where
$\Sp(\Mrm)$ is the spectrum of $\Mrm$.  We can observe three different
behaviours of the proposal given by \eqref{eq:bO} when $x$ is large,
which are implied by the behaviour of $\lambdaHessianfMin)$ and
$\lambdaHessianfMax$.
\\

If $\liminf_{\norm{x} \to \plusinfty} \lambdaHessianfMin =0$. Then, $g(x) = o(\norm{x}^2)$ as $\|x\|\rightarrow \infty$, and $y^{\bO}$ tends to be as the
MALA proposal at infinity, and we can show that bOMA is geometrically ergodic with the same conditions
introduced in \cite{RT96} for this one.
\begin{example}
  By \cite[Theorem 4.1]{RT96}  bOMA is  geometrically ergodic for $\pi
  \in \Escr(\gamma, \beta)$ with  $\beta \in \coint{1,2}$.
\end{example}

 Now, we focus on the case where $\limsup_{\norm{x} \to \plusinfty} \lambdaHessianfMax <0$. For instance, this condition holds
for $\pi \in \Escr(\gamma,\beta)$ when $\beta \geq 2$.
We give conditions similar to the one for geometric convergence of the Ozaki
discretization, given in \cite{hansen:2003}, to check conditions
of Lemma \ref{propo:convergence_gaussian_proposal}. Although these conditions does not cover all the cases, they seem
to apply to interesting ones. Here are our assumptions where we denote by $\mathbb{S}^{d} = \{x \in \rset^d, \eqsp  \norm{x} =1\}$, the sphere in $\rset^d$ and $\normalX = x / \norm{x}$.

\begin{assumption}
  \label{hyp:bo_contraction}
	We assume:
\begin{enumerate}
  \item
    \label{item:hyp:bo_contraction_limsup}
    $\limsup_{\norm{x}\to \plusinfty} \lambdaHessianfMax <0$;
    \item
   \label{item:hyp:bo_contraction_limThird}
   $\lim_{\norm{\varx} \to \plusinfty} Df(\varx)^{-2} \{ \Id:D^2f(\varx) \} = 0$;
   \item
     \label{item:hyp:bo_contraction_asympt}
$Df(\varx)^{-1} f(\varx)$ is asymptotically homogeneous to $\varx$ when
     $\norm{\varx} \to \plusinfty$, \ie~ there exists a function $c: \mathbb{S}^{d} \to
     \IR$ such that
     \[
    \lim_{\norm{\varx} \to \plusinfty} \norm{\frac{Df(\varx)^{-1}  f(\varx)}{\norm{\varx}} -c(\normalX) \normalX} = 0 \eqsp.
    \]
  \end{enumerate}
  \end{assumption}

  The condition \ref{item:hyp:bo_contraction_limsup} in Assumption~\ref{hyp:bo_contraction} implies that  for all $\varx \in \IR^d$,
  $\lambdaMax{Df(\varx)} \leq M_f$, and garantees that $S^{\bO}(\varx,h)$ is bounded for all
  $\varx \in \IR^d$.

  \begin{lemma}
    \label{lem:bound_Sigma_bO}
Assume Assumptions \ref{hyp:density_convergence} and \ref{hyp:bo_contraction}.   There exists $M_\Sigma \geq 0$ such that for all $\varx \in \IR^d$
   $\norm{S^{\bO}(\varx,h)} \leq M_\Sigma$.
    \end{lemma}
\begin{proof}
  Since $S^{\bO}(\varx,h)$ is symmetric for all $\varx \in \IR^d$, and $t \mapsto
  (\rme^{ht} - 1)/t + (1/3) (\rme^{-(ht)^2}-1)/t$ is bounded on $\ocint{-\infty, M}$ for
  all $M \in \IR$, we just need to show that there exists $M_f \geq 0$ such that for all
  $\varx$, $\lambdaMax{Df(\varx)} \leq M_f$.  First, by Assumption
  \ref{hyp:bo_contraction}-\eqref{item:hyp:bo_contraction_limsup}, there exists $R
  \geq 0$, such that for all $\varx, \norm{x} \geq R$, $\Sp (Df(\varx))
  \subset \IR_-$. In addition by Assumption \ref{hyp:density_convergence} $\varx \mapsto
  Df(\varx)$ is continuous, and there
  exists $M \geq 0$ such that for all $\varx, \norm{x} \leq R$, $\norm{Df(\varx)} \leq
  M$.
\end{proof}

\begin{theorem}
 \label{propo:sufficient_condition_contraction_oB}
Assume Assumptions \ref{hyp:density_convergence}, \ref{hyp:ratio_convergence_bO} and \ref{hyp:bo_contraction}. If
\begin{equation}
\label{eq:condition_c}
0 < \inf_{n \in \mathbb{S}^{d}} c(n) \leq \sup_{n \in \mathbb{S}^{d}} c(n) < 6/5 \eqsp,
\end{equation}
then bOMA is geometrically ergodic.
\end{theorem}

\begin{proof}
We check that the conditions of Lemma \ref{propo:convergence_gaussian_proposal} hold. By Assumption \ref{hyp:density_convergence} and \eqref{eq:bO}, Assumption \ref{hyp:Gaussian_Markov_kernel} holds, thus it remains to check \eqref{eq:hyp:gaussian_prop_mu}.
First, Lemma
\ref{lem:bound_Sigma_bO} implies that the second equality of \eqref{eq:hyp:gaussian_prop_mu} is satisfied, and we just need to prove the first equality. By \cite[Lemma 3.4]{hansen:2003}, it suffices  to prove that
\begin{equation}
  \label{eq:proof:propo:sufficient_condition_contraction_oB_1}
\limsup_{\norm{\varx} \to \plusinfty} \ps{\frac{\eta(\varx)}{\norm{\varx}}}{\frac{\eta(\varx)}{\norm{\varx}} + 2
  \normalX} < 0 \eqsp,
\end{equation}
where $\eta(x) = \mu^{\bO}(\varx,h) - \varx$.
Since $\limsup_{\norm{x} \to \plusinfty} \lambdaHessianfMax <0$ we can write $\Gscr(\varx)=\Bscr(\varx)Df(\varx)^{-1}  f (\varx)$, where
\begin{equation*}
\Bscr(\varx) =(\rme^{(h/2)  Df(\varx)} - \Id) + (2/3)(\rme^{-(h Df(\varx)  /2)^2}- \Id) \eqsp,
\end{equation*}
and $x \mapsto \Bscr(\varx)$ is bounded on $\rset^d$.  Since $\Bscr$ is bounded on $\rset^d$, by Assumption
\ref{hyp:bo_contraction}-\eqref{item:hyp:bo_contraction_limThird}-\eqref{item:hyp:bo_contraction_asympt}
and (\ref{eq:condition_c}),
\begin{equation}
  \label{eq:proof:propo:sufficient_condition_contraction_oB_2}
\lim_{\norm{x} \to \plusinfty} \abs{\ps{\frac{\eta(\varx)}{\norm{\varx}}}{\frac{\eta(\varx)}{\norm{\varx}} + 2
  \normalX} - \norm{\Bscr(\varx)\normalX}^2 c(\normalX)^2 + 2 \ps{\Bscr(\varx)\normalX}{\normalX} c(\normalX)
} = 0  \eqsp.
\end{equation}
In addition, if we denote  the eigenvalues of $\Bscr(\varx)$  by $\{\lambda_i(\varx) , \eqsp i =1, \dots, d\}$ and $\{e_i(\varx) , \eqsp i =1, \dots, d\}$ an
orthonormal basis of eigenvectors, we have
\begin{multline}
\label{eq:proof_sufficent_condition_bO}
\norm{\Bscr(\varx)\normalX}^2 c(\normalX)^2 + 2 \ps{\Bscr(\varx)\normalX}{\normalX}
c(\normalX)
\\=  \sum_{i =1} ^d c(\normalX)\lambda_i(\varx) \ps{e_i(\varx)} {\normalX}^2
(  c(\normalX)\lambda_i(\varx)  + 2 )
\end{multline}
Since $\limsup_{\norm{x} \to \plusinfty} Df(x) < 0$, for all $i$ and $\norm{x}$ large
enough, $\lambda_i(\varx) \in \coint{-5/3,0}$. Therefore using
\eqref{eq:condition_c} we get from \eqref{eq:proof_sufficent_condition_bO}:
\[
\norm{\Bscr(\varx)\normalX}^2 c(\normalX)^2 + 2 \ps{\Bscr(\varx)\normalX}{\normalX} c(\normalX) < 0 \eqsp.
\]
The proof is concluded using this result in \eqref{eq:proof:propo:sufficient_condition_contraction_oB_2}.
\end{proof}



\subsubsection*{Application to the convergence of bOMA for $\pi \in \Pscr_m^+$}
For the proof of the main result of this section, we need the following lemma.
\begin{lemma}[\protect{\cite[Proof of Theorem 4.10]{hansen:2003}}]
  \label{lem:hansen_asympt_poly_homo}
Let $\pi \in \Pscr_m ^+$ for $m \geq 2$, then
  $\pi$ satisfies Assumption~\ref{hyp:bo_contraction}-\eqref{item:hyp:bo_contraction_asympt} with
  $c(\nArrow) = 1/(m-1) \in \ooint{0,6/5}$ for all $\nArrow \in \mathbb{S}^{d}$.
\end{lemma}

\begin{proposition} \label{pro:same}
Let $\pi \in \Pscr_m^+$ for $m \geq 2$, then bOMA is $V$-geometrically ergodic, with $V(x)
=\norm{x}^2+1$.
\end{proposition}

\begin{proof}
  Let us denote $\pi \propto \exp(-\prm(\varx) - \rrm(\varx))$, with $\prm$ and $\rrm$ satisfying
  the conditions from the definition in Section \ref{sec:poly_homo}.  We prove that if $\pi
  \in \Pscr_m^+$, Theorem \ref{propo:sufficient_condition_contraction_oB} can be applied.
%
    First,  by definition of $\Pscr_m^+$, Assumption \ref{hyp:density_convergence} is satisfied.
 Furthermore, Assumption  \ref{hyp:bo_contraction}-\eqref{item:hyp:bo_contraction_limsup}-\eqref{item:hyp:bo_contraction_limThird} follows from  \eqref{eq:petit_to_r_poly_homog}, \eqref{eq:hessian_poly_homog},
  \eqref{eq:third_derivatives_poly_homog} and the condition that $\nabla^2 \prm (\nArrow)$ is
  positive definite for all $\nArrow \in \mathbb{S}^{d}$. Also by Lemma \ref{lem:hansen_asympt_poly_homo},  Assumption
  \ref{hyp:bo_contraction}-\eqref{item:hyp:bo_contraction_asympt} is satisfied.\\
  Now we focus on Assumption \ref{hyp:ratio_convergence_bO}.
  For ease of notation, in the following  we denote  $\mu^{\bO}$ and $S^{\bO}$ by $\mu$ and $S$ , and do not mention  the dependence in the parameter $h$ of $\mu$ and $S$ when it does not play any role. Note that
  \begin{equation}
  \label{eq:expression_ratio_accept}
  \int_{\IR^d} \alpha( \varx , \vary) q(\varx,\vary) \rmd \vary = (2 \uppi)^{-d/2}\int_{\IR^d} \defSet{1 \wedge
  \exp \widetilde{\alpha}(\varx,\GaussianRV) }\exp(-\norm{\GaussianRV}^2/2) \rmd \GaussianRV \eqsp,
  \end{equation}
  where
\begin{align}
    \nonumber
    \widetilde{\alpha}(\varx,\GaussianRV) &= -\prm(\mu(\varx)
  + S(\varx)\GaussianRV) + \prm(\varx)-\rrm(\mu(\varx) \\
  \nonumber
  &+ S(\varx)\GaussianRV) + \rrm(\varx)-
  \log(\abs{S(\mu(\varx) + S(\varx)\GaussianRV) })  +  \log(\abs{S(\varx) }) + (1/2) \norm{\GaussianRV}^2 \\
    \label{eq:ratio_expression_poly_homo}
  & -(1/2) \ps{(\widetilde{S}(x, \GaussianRV))^{-1} \{ \varx - \mu(\mu(\varx) +
      S(\varx)\GaussianRV) \}}{\varx - \mu \left(\mu(\varx) +
      S(\varx)\GaussianRV\right) } \eqsp,
\end{align}
and $\widetilde{S}(x,\GaussianRV) =S(\mu(\varx) + S(\varx)\GaussianRV) S(\mu(\varx) + S(\varx)\GaussianRV)^T$.
First, we consider $m \geq 3$, then we have the following estimate of the terms in
\eqref{eq:ratio_expression_poly_homo} by
\eqref{eq:petit_to_r_poly_homog}-\eqref{eq:third_derivatives_poly_homog} and Lemma \ref{lem:hansen_asympt_poly_homo}:
\begin{align}
\label{eq:proof_poly_homo_mu_mgeq3}
\mu(\varw) &\underset{\norm{w} \to \plusinfty}{=} \{1-5/(3(m-1))\}w + o(\norm{\varw}) \\
\label{eq:proof_poly_homo_sigma_mgeq3}
(S(\varw)S(\varw)^T)^{-1} &
 \underset{\norm{w} \to \plusinfty}{=}\frac{3}{4}
  m(m-1)\norm{\varw}^{m-2} \nabla^2 \prm (\nArrow_{\varw}) + o(\norm{\varw}^{m-2})
\\
\label{eq:proof_poly_homo_log_mgeq3}
\log(\abs{S(\varw)}) &\underset{\norm{w} \to \plusinfty}{=} o(\norm{\varw})
\end{align}
Then by \eqref{eq:proof_poly_homo_mu_mgeq3}-\eqref{eq:proof_poly_homo_log_mgeq3}, if we define $\Psi: \coint{3,\plusinfty} \rightarrow \IR$ by
\begin{multline*}
m \mapsto 1-\defSet{1-\frac{5}{3(m-1)}}^m  \\ - (3/8)m(m-1) \defSet{
1-\bracket{1-\frac{5}{3(m-1)}}^2 }^2 \defSet{1-\frac{5}{3(m-1)} }^{m-2}  \eqsp,
\end{multline*}
we get
\begin{equation*}
\widetilde{\alpha}(\varx,\GaussianRV) \underset{\norm{x} \to \plusinfty}{=} \norm{\varx}^m
\prm(\normalX) \Psi(m)
+ o(\norm{\varx}^m) \eqsp.
\end{equation*}
Since $\Psi$ is positive on $\coint{3,\plusinfty}$, for all $\GaussianRV \in \rset^d$ $\lim_{\norm{\varx} \to \plusinfty} \widetilde{\alpha}(\varx
,\GaussianRV) = \plusinfty$. This result, \eqref{eq:expression_ratio_accept} and  Fatou's Lemma  imply that Assumption \ref{hyp:ratio_convergence_bO} is
satisfied.\\
For $m = 2$, we can assume $\prm(\varx) = \ps{\Arm \varx}{\varx}$ with $\Arm \in \mathcal{S}^*_+(\rset^d)$. Let us denote
for $\Mrm$ an invertible matrix of dimension $p \geq 1$,
\begin{align*}
\varrho(\Mrm) &= (\rme^{-  \Mrm} - \operatorname{I}_p) + (2/3)(\rme^{- \Mrm^2}- \operatorname{I}_p) \\
\varsigma(\Mrm) &= (\rme^{-2 \Mrm }-\operatorname{I}_p) +(1/3)(\rme^{-4 \Mrm ^2}-\operatorname{I}_p) \eqsp.
\end{align*}
Then we have the following estimates: 
\begin{align}
  \nonumber
\widetilde{\alpha}(\varx,\GaussianRV)& \underset{\norm{x} \to \plusinfty}{=} \ps{\Arm(\varsigma(h\Arm))^{-1}\defSet{\bracket{2 \varrho(h\Arm) + \varrho(h\Arm)^2}\varx}}
{\bracket{2 \varrho(h\Arm) + \varrho(h\Arm)^2}\varx}\\
\label{eq:proof_poly_homog_log_ratio_case_quad}
& +\ps{ \Arm \varx}{\varx} -\ps{
\Arm \defSet{ (\Id + \varrho(h\Arm))\varx }}{(\Id + \varrho(h\Arm))\varx}   + o(\norm{\varx}^2)
\end{align}
If we denote the eigenvalues of $\Arm$ by  $\{\lambda_i, i = 1 \dots d \}$ and $\{x_i , i = 1 , \dots, d \}$ the coordinates of $\varx$ in
an orthonormal  basis  of eigenvectors for  $\Arm$, \eqref{eq:proof_poly_homog_log_ratio_case_quad}
becomes
\begin{align}
  \label{eq:proof_poly_homog_log_ratio_case_quad2}
&\widetilde{\alpha}(\varx,\GaussianRV) \underset{\norm{x} \to \plusinfty}{=}   \sum_{i=1} ^d  \Xi(h,\lambda_i)  x_i^2   + o(\norm{\varx}^2) \eqsp.
\end{align}
where for $h, \lambda >0$,
\[
\Xi(h, \lambda) =  \lambda \bracket{ 1 - (\varrho(h \lambda)+1)^2 +
  \varsigma(h \lambda)^{-1}\bracket{4\varrho(h \lambda)^2 + 4 \varrho(h \lambda)^3 +
    \varrho(h \lambda)^4} }\eqsp.
\]
Using that for any $h,\lambda >0$, $ \Xi(h, \lambda) > 0$ and \eqref{eq:proof_poly_homog_log_ratio_case_quad2}, we have for all $\GaussianRV \in \rset^d$,
$\lim_{\norm{x} \to \plusinfty} \widetilde{\alpha}(\varx,\GaussianRV) =
\plusinfty$, and as in the first case Assumption \ref{hyp:ratio_convergence_bO} is satisfied.
\end{proof}

\begin{remark}
  Using the same reasoning as in Proposition \ref{pro:same}, one can show that bOMA is geometrically ergodic for $\pi \in \Escr(\beta, \gamma)$ with $\beta \geq 2$.
\end{remark}

We now summarise the behaviour for all the different algorithms for the one dimensional class  $\Escr(\beta,\gamma)$ in Table  \ref{tb:summary_2}
\begin{table}[h!]
\begin{center}    \begin{tabular}{| p{2.5cm} | p{2.5cm} | p{2.8cm}|p{2.5cm} | }
    \hline
     Method &  $\beta \in [1,2)$ &  $\beta=2 $ &  $\beta>2$ \\   \hline \hline
 fMALA  \eqref{eq:fMALA}   & geometrically
 ergodic & geometrically ergodic or not & not geometrically ergodic\\  \hline
 mOMA \eqref{eq:mO} 	    &   geometrically ergodic & geometrically ergodic or not  & not geometrically ergodic\\  \hline
 bOMA \eqref{eq:bO}   &  geometrically ergodic & geometrically ergodic & geometrically ergodic\\  \hline
    \end{tabular}
    \end{center}
  \caption{Summary of ergodicity results for the Metropolis-Hastings algorithms  for the class $\Escr(\beta,\gamma)$}
\label{tb:summary_2}
\end{table}

\subsection{Convergence of Gaussian Markov kernel on $\rset$}
\label{subsec:convergence_unadjusted}
We now present precise results for the ergodicity of the unadjusted proposals, by  extending the results of \cite{RT96} for the ULA to Gaussian Markov kernels on $\rset$.
Under Assumption~\ref{hyp:Gaussian_Markov_kernel}, it is straightforward to see that $Q$ is
$\Leb^d$-irreducible, where $\Leb^d$ is the Lebesgue measure,
aperiodic and all compact set of $\rset^d$ are small; see
\cite[Theorem 3.1]{hansen:2003}.
We  now state our main theorems, which essentially complete \cite[Theorem
3.1-3.2]{RT96}. Since their proof are very similar, they are omitted.
\begin{theorem}
\label{theo_convergence_Gaussian_Markov_kernel_S_bounded}
Assume Assumption \ref{hyp:Gaussian_Markov_kernel}, and there exist $s_\wedge,u_+,u_- \in \rset_+^*$ and $\exponentTheo \in \rset$ such that:
\begin{equation*}
\limsup_{\abs{x} \to \plusinfty} S(x) \leq s_\wedge \eqsp,
 \end{equation*}
 \begin{equation*}
  \lim_{x \to \plusinfty} \defSet{\mu(x) - x}x^{-\exponentTheo} = -u_+ \eqsp, \text{ and }
\lim_{x \to -\infty} \defSet{\mu(x) - x}\abs{x}^{-\exponentTheo} = u_- \eqsp.
\end{equation*}
\begin{enumerate}[(1)]
\item
\label{item:theo_convergence_Gaussian_Markov_kernel_S_bounded_1}
If $\exponentTheo \in \coint{0,1}$, then $Q$ is geometrically ergodic.
\item
\label{item:theo_convergence_Gaussian_Markov_kernel_S_bounded_2}
If $\exponentTheo =1$ and $(1-u_+)(1-u_-) <1$, then $Q$ is geometrically ergodic.
\item
\label{item:theo_convergence_Gaussian_Markov_kernel_S_bounded_3}
If $\exponentTheo \in \ooint{-1,0}$, then $Q$ is ergodic but not geometrically ergodic.
\end{enumerate}
\end{theorem}

\begin{proof}
See the proof of \cite[Theorem 3.1]{RT96}.
\end{proof}

\begin{theorem}
\label{theo_convergence_Gaussian_Markov_kernel_S_unbounded}
Assume Assumption \ref{hyp:Gaussian_Markov_kernel}, and there exist $s_\vee,u_+,u_- \in \rset_+^*$ and $\exponentTheo \in \rset$ such that:
\begin{equation*}
\liminf_{\abs{x} \to \plusinfty} S(x)\geq s_\vee \eqsp,
 \end{equation*}
 \begin{equation*}
\lim_{x \to \plusinfty} S(x)^{-1}\mu(x)x^{-\exponentTheo} = -u_+ \eqsp, \text{ and }
\lim_{x \to -\infty}  S(x)^{-1}\mu(x)\abs{x}^{-\exponentTheo} = u_- \eqsp.
\end{equation*}
\begin{enumerate}[(1)]
\item
\label{item:theo_convergence_Gaussian_Markov_kernel_S_unbounded_1}
If $\exponentTheo >1$, then $Q$ is transient.
\item
\label{item:theo_convergence_Gaussian_Markov_kernel_S_unbounded_2}
If $\exponentTheo =1$ and $(u_+ \wedge u_-) s_\vee >1$, then $Q$ is transient.
\end{enumerate}
\end{theorem}

\begin{proof}
See the proof of \cite[Theorem 3.2]{RT96}.
\end{proof}

\subsubsection*{Ergodicity of the unadjusted proposals for the class $\Escr(\beta,\gamma)$}
\label{subsec:4.4}
We now apply Theorems \ref{theo_convergence_Gaussian_Markov_kernel_S_bounded} and \ref{theo_convergence_Gaussian_Markov_kernel_S_unbounded}
in order to study the ergodicity of the  different unadjusted proposals applied to $\pi \in \Escr(\beta,\gamma)$.  In the case $\beta \in \ooint{0,2}$ all the three algorithms (fULA,mUOA,bUOA) have their mean map behaving like
$x -\beta \gamma x \abs{x}^{\beta-2}/2$ at infinity and their variance map bounded from above. This is exactly the behaviour that ULA \cite{RT96} has for the same values of $\beta$, thus it should not be a surprise that Theorem \ref{theo_convergence_Gaussian_Markov_kernel_S_bounded} implies that all the three algorithms behaved as the ULA does for the corresponding values, namely being ergodic for $\beta \in \ooint{0,1}$ and geometrically ergodic for $\beta \in \coint{1,2}$. Furthermore, for values of $\beta \geq 2$ we have the following cases.
\begin{enumerate}[(a)]
\item For $\beta = 2$,
\begin{enumerate}[-]
\item  fULA is geometrically ergodic if $h\gamma(1+h \gamma /6) \in \ooint{0,2}$ by Theorem
  \ref{theo_convergence_Gaussian_Markov_kernel_S_bounded}-\eqref{item:theo_convergence_Gaussian_Markov_kernel_S_bounded_2},
  and is transcient if $h\gamma(1+h \gamma /6) >2$ by Theorem
  \ref{theo_convergence_Gaussian_Markov_kernel_S_unbounded}-\eqref{item:theo_convergence_Gaussian_Markov_kernel_S_unbounded_2},
  since $\mu^{\fMALA}$ is equivalent at infinity to $(1-h \gamma(1+h
  \gamma /6))x$ and $S^{\fMALA}(x)$ is constant for $\abs{x} \geq
  R_\pi$.
\item
 mUOA is geometrically ergodic if $1+2(h\gamma)^2/3 -\rme^{-\gamma h}\in
  \ooint{0,2}$ by Theorem
  \ref{theo_convergence_Gaussian_Markov_kernel_S_bounded}-\eqref{item:theo_convergence_Gaussian_Markov_kernel_S_bounded_2},
  and is transcient if $ 1+2(h\gamma)^2/3 -\rme^{-\gamma h} >2$ by Theorem
  \ref{theo_convergence_Gaussian_Markov_kernel_S_unbounded}-\eqref{item:theo_convergence_Gaussian_Markov_kernel_S_unbounded_2},
  since $\mu^{\mO}$ is equivalent at infinity to $(\rme^{-\gamma h}-2(h\gamma)^2/3 )x$ and $S^{\mO}(x)$ is constant for $\abs{x} \geq
  R_\pi$.
\item bUOA is geometrically ergodic by Theorem
  \ref{theo_convergence_Gaussian_Markov_kernel_S_bounded}-\eqref{item:theo_convergence_Gaussian_Markov_kernel_S_bounded_2},
  since $\mu^{\bO}$ is equivalent at infinity to $-2x/3$
  and $S^{\bO}(x)$ is constant for $\abs{x} \geq R_\pi$.
\end{enumerate}
\item
For $\beta >2$,
\begin{enumerate}[-]
\item fULA and mUOA are transcient by Theorem
  \ref{theo_convergence_Gaussian_Markov_kernel_S_unbounded}-\eqref{item:theo_convergence_Gaussian_Markov_kernel_S_unbounded_1}
  since their mean value map is equivalent at infinity to $-C_1 \abs{x}^{2 \beta -2}/x$,
  and their variance map to $C_2\abs{x}^{\beta-2}$ for some constants $C_1,C_2 >0$, and
  their variance map are bounded from below.
\item
bUOA is geometrically ergodic by Theorem \ref{theo_convergence_Gaussian_Markov_kernel_S_bounded}-\eqref{item:theo_convergence_Gaussian_Markov_kernel_S_bounded_1} since its mean value map is equivalent at infinity to $\defSet{1-5/(3(\beta-1))}x$ and its variance map is bounded from above.
\end{enumerate}
\end{enumerate}

The summary of our findings can be found in Table
\ref{tb:summary_1}.

\begin{table}[th]
\begin{center}    \begin{tabular}{| p{2.cm} | p{2.cm} | p{2.cm} | p{2.35cm}|p{2.cm} | }
    \hline
     Method & $\beta \in (0,1)$  & $\beta \in [1,2)$ &  $\beta=2 $ &  $\beta>2$ \\   \hline \hline
 fULA  \eqref{eq:fMALA}   & ergodic & geometrically ergodic & geometrically ergodic/transient & transient\\  \hline
 mUOA \eqref{eq:mO} 	    & ergodic&  geometrically ergodic & geometrically ergodic/transient  & transient\\  \hline
 bUOA \eqref{eq:bO}   & ergodic& geometrically ergodic & geometrically ergodic & geometrically ergodic\\  \hline
    \end{tabular}
    \end{center}
\caption{Summary of ergodicity results for the unadjusted proposals for the class $\Escr(\beta,\gamma)$.}
\label{tb:summary_1}
\end{table}



\section{Numerical illustration of the improved efficiency}
\label{sec:numer}

\begin{figure}[t!]
\centering
 \scalebox{0.70}{\global\def\path{#1}\input{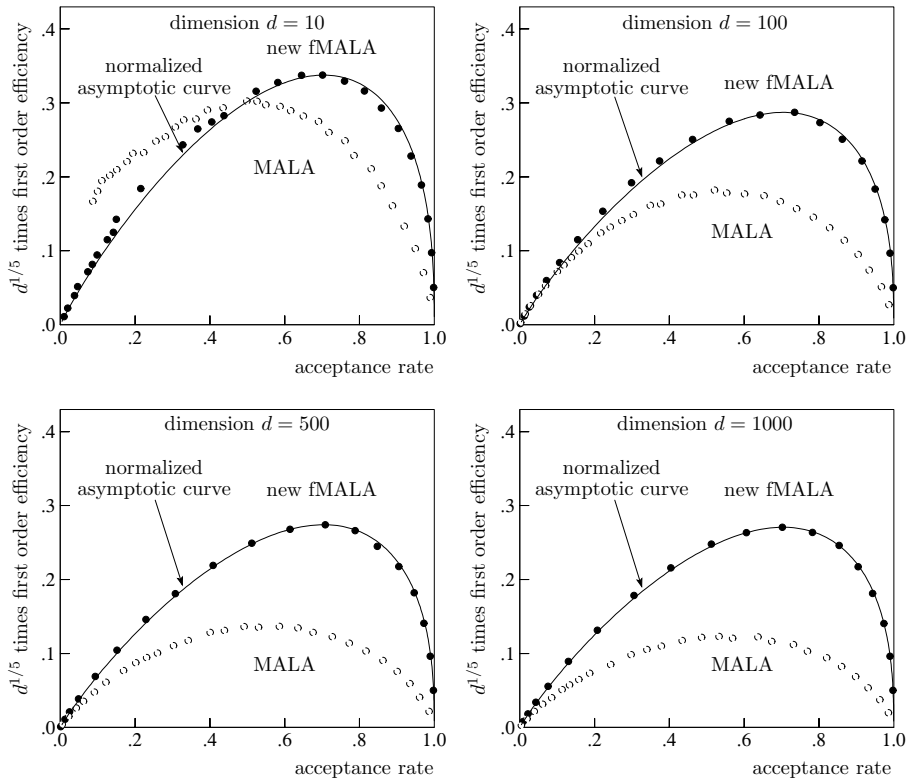}}
 \bigskip
 \caption{ First-order efficiency of the new fMALA and the standard
   MALA for the double well potential $g(x)=-\frac14x^4+\frac12x^2$,
   as a function of the overall acceptance rates in dimensions
   $d=10,100,500,1000$.  The solid line is the
   reference asymptotic curve of efficiency for the new fMALA,
   normalised to have the same maximum value as the finite dimensional
   fMALA.
\label{fig:figefficiency}}
\end{figure}

In this section, we illustrate our analysis (Section \ref{ssec:optimal_scaling_fMALA}) of the asymptotic behaviour of fMALA as
the dimension $d$ tends to infinity, and we demonstrate its gain of efficiency as $d$ increases compared to the standard MALA.
Following \cite{RR98}, we define the first-order efficiency of a multidimensional Markov chain $\{X_k, k \in \nset\}$ with first component denoted $X_k^{(1)}$ as $\IE[(X_{k+1}^{(1)}-X_k^{(1)})^2]$.
In Figure \ref{fig:figefficiency}, we consider as a test problem the product case \eqref{eq:prod}
using the double well potential with $g(x)=-\frac14x^4+\frac12x^2$ in dimensions
 $d=10,100,500,1000$, respectively.
 We consider many time stepsizes $h=\ell^2d^{-1/5}$, plotting the first order efficiency (multiplied by $d^{1/5}$ because this is the scale which is asymptotically constant for fMALA as $d\rightarrow \infty$) as a function of the acceptance rate for the standard MALA (white bullets) and the acceptance rate $\acceptMean_d^{\fMALA}(\ell)$ of the improved version fMALA (black bullets), respectively. For simplicity, each chain is started from the origin.
The expectations are approximated as the average over $2\times 10^5$ iterations of the algorithms and we use the same sets of generated random numbers for both methods.
For comparison, we also include (as solid lines) the asymptotic efficiency curve of fMALA as $d$ goes to infinity, normalised to have the same maximum as fMALA in finite dimension $d$. This corresponds to the (rescaled) limiting diffusion speed $\speedMeasure^{\fMALA} (\ell)$ as a function of $\acceptMean^{\fMALA} (\ell)$ (quantities given respectively in Theorems \ref{theo:accept_ratio_fMALA} and \ref{theo:scaling:fMALA}).
We observe excellent agreement of the numerical first order efficiency
compared to the asymptotic one, especially as $d$ increases, which corroborates the scaling results of fMALA. 
In addition, we observe for the considered dimensions $d$ that the optimal acceptance rate maximizing the first-order efficiency remains very close to the limiting value of 0.704 predicted in Theorem \ref{theo:scaling:fMALA}.
This numerical experiment shows that the efficiency improvement of fMALA compared to MALA
is significant and indeed increases as the dimension $d$ increases, which confirms the analysis of Section \ref{ssec:optimal_scaling_fMALA}.

\begin{figure}[t!]
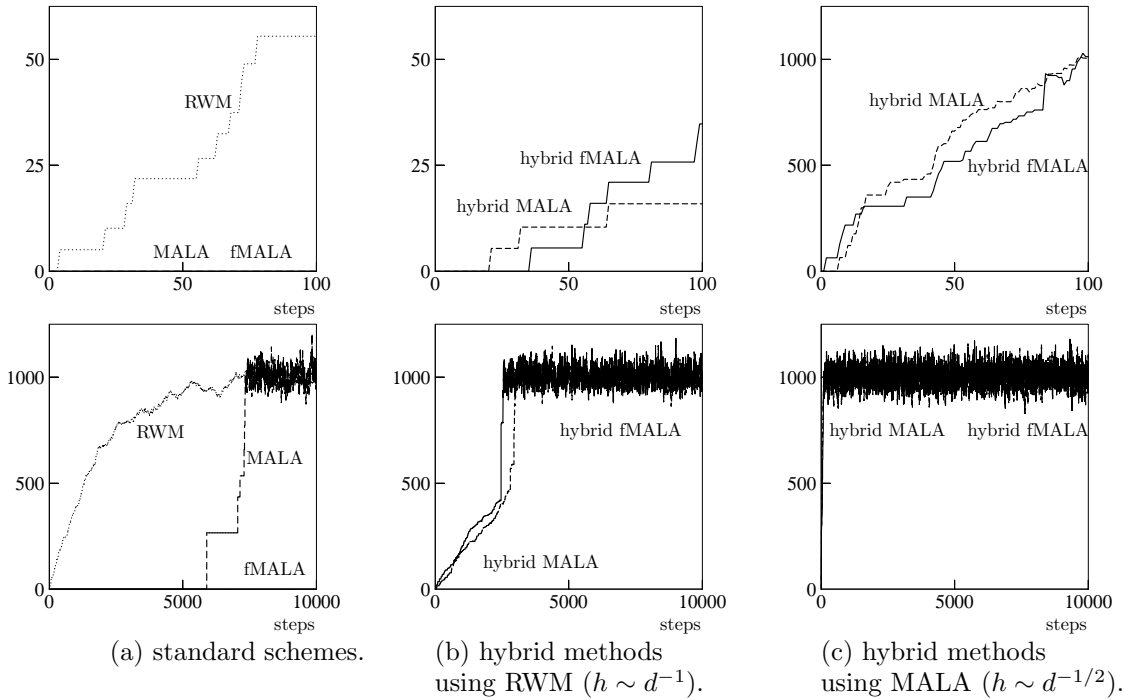

 \centering \small
	\begin{subfigure}[t]{0.32\textwidth}
  \scalebox{0.70}{\global\def\path{#1}\input{figtails/figtailsa.inp}}
	\bigskip
	\caption{\small standard schemes. }
	\label{fig:figtailsa}
	\end{subfigure}
	\begin{subfigure}[t]{0.32\textwidth}
  \scalebox{0.70}{\global\def\path{#1}\input{figtails/figtailsb.inp}}
	\bigskip
	\caption{\small hybrid methods \newline using RWM ($h\sim d^{-1}$). }
	\label{fig:figtailsb}
	\end{subfigure}
	\begin{subfigure}[t]{0.32\textwidth}
  \scalebox{0.70}{\global\def\path{#1}\input{figtails/figtailsc.inp}}
	\bigskip
	\caption{\small hybrid methods \newline  using MALA ($h\sim d^{-1/2}$). }
	\label{fig:figtailsc}
	\end{subfigure}
 \caption{
Trace plots of $\|X\|^2$ for the Gaussian target density in dimension $d=1000$ when
starting at the origin.
 Comparison of fMALA with $h\sim d^{-1/5}$ (solid lines), MALA with $h\sim d^{-1/3}$ (dashed lines), RWM with $h\sim d^{-1}$ (dotted lines).
\label{fig:figtails}
 }
 \end{figure}

 \begin{figure}[t!]
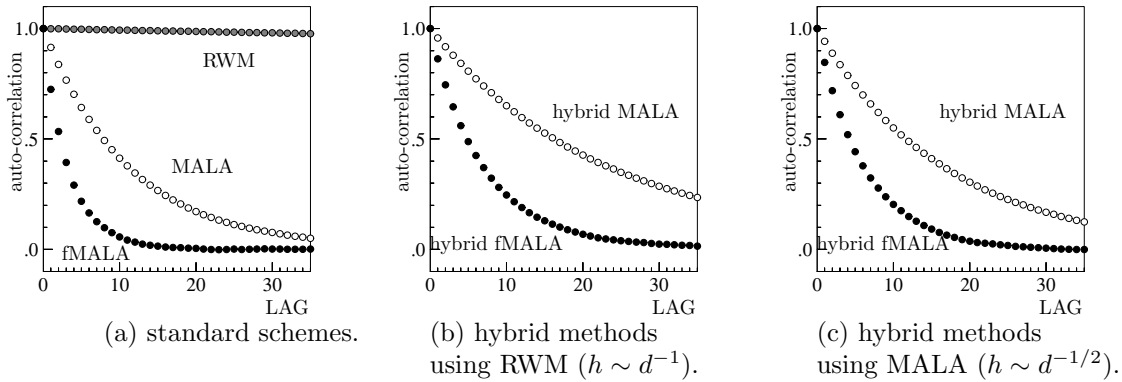

 \centering
	\begin{subfigure}[t]{0.32\textwidth}
  \scalebox{0.70}{\global\def\path{#1}\input{figcorrelation/figcorrelationa.inp}}
	\bigskip
	\caption{\small standard schemes. }
	\label{fig:figcorrelationa}
	\end{subfigure}
	\begin{subfigure}[t]{0.32\textwidth}
  \scalebox{0.70}{\global\def\path{#1}\input{figcorrelation/figcorrelationb.inp}}
	\bigskip
	\caption{\small hybrid methods \newline using RWM ($h\sim d^{-1}$). }
	\label{fig:figcorrelationb}
	\end{subfigure}
	\begin{subfigure}[t]{0.32\textwidth}
  \scalebox{0.70}{\global\def\path{#1}\input{figcorrelation/figcorrelationc.inp}}
	\bigskip
	\caption{\small hybrid methods \newline  using MALA ($h\sim d^{-1/2}$). }
	\label{fig:figcorrelationc}
	\end{subfigure}
 \caption{
 Auto-correlation versus LAG for the Gaussian target density in dimension $d=1000$.
 Comparison of fMALA with $h\sim d^{-1/5}$ (black), MALA with $h\sim d^{-1/3}$ (white), RWM with $h\sim d^{-1}$ (gray).
\label{fig:figcorrelation}
 }
 \end{figure}

For our next experiments, we consider the $d$-dimensional
  zero-mean Gaussian distribution with covariance matrix $\Id$ for $d
  = 1000$, as target distribution. We aim to numerically study the
  transient behaviour of fMALA and propose some solutions to overcome
  this issue. In Figure~\ref{fig:figtails}, we plot the squared norm of $10^4$
  samples generated by the RWM, MALA, fMALA and some hybrid strategies
  for MALA and fMALA, all started from the origin.  
We also include a zoom on the first 100 steps.
In Figure~\ref{fig:figtailsa}, we use standard implementations of the
schemes. The time step $h$ for each algorithm is chosen as the optimal
parameter based on the optimal scaling results of all the algorithms
at stationarity: for the RWM $h=2.38^2 d^{-1}$, for MALA $h= 1.65^2
d^{-1/3}$ and for fMALA $h= 1.79^2 d^{-1/5}$. It can observed that
MALA exhibits many rejected steps in contrast to RWM.  This is a known
issue of MALA in the transient phase \cite{christensen2005scaling,jourdain2014optimal} due to a tiny acceptance
probability at first steps, and the same behaviour can be observed for
fMALA, with zero accepted step in the present simulation.  To
circumvent this issue, the following hybrid MALA
scheme was presented in \cite{christensen2005scaling}.  The idea is to 
combine MALA with RWM at each step: with probability $1/2$, we apply
the MALA proposal \eqref{eq:MALA} with step size $h = 1.65^2
d^{-1/3}$, the optimal parameter for MALA at stationarity. Otherwise,
the RWM proposal \eqref{eq:RWM} is used with step size $h=2.38^2
d^{-1}$, the optimal parameter for the RWM at stationarity.  Indeed,
\cite{christensen2005scaling} and \cite{jourdain2014optimal} have
shown that the optimal scaling in the transient phase and at
stationarity is the same and scales as $d^{-1}$. In
Figure~\ref{fig:figtailsb}, the plots for this hybrid MALA are
presented, the same methodology is also applied for the hybrid
fMALA scheme, showing a behaviour similar to hybrid MALA.  In
Figure~\ref{fig:figtailsc}, the RWM proposal is replaced by the
MALA proposal \eqref{eq:MALA} with a different step size
$h=2d^{-1/2}$, which is the optimal parameter for MALA in the
transient phase according to \cite{christensen2005scaling}.  Again,
hybrid fMALA exhibits a behaviour similar to hybrid MALA.

In Figure~\ref{fig:figcorrelation}, we consider again the same
  schemes and hybrid versions as in Figure~\ref{fig:figtails}, with
  the same step sizes, and we compare their autocorrelation
  function. We consider for each algorithms $2\cdot 10^5$ iterations started at stationarity,
  where the first $10^{3}$ iterations were
  discarded as burn-in.  In Figure~\ref{fig:figcorrelationa}, it can
  be observed that the autocorrelation associated with fMALA goes to
  $0$ quicker than the RWM and MALA.  In
  Figure~\ref{fig:figcorrelationb}, and
  Figure~\ref{fig:figcorrelationc}, we observe that by using hybrid strategies
  which are designed to robustify convergence from the transient phase,
  fMALA still comfortably outperforms MALA in terms of expected square efficiency
  (which is a stationary quantity).

Although our analysis applies only to product measure densities of the form \eqref{eq:prod},
we next consider the following non-product density in $\mathbb{R}^d$, defined using a normalization constant $Z_d$ and for $X_0=0$ as
\begin{equation} \label{eq:rhononproduct}
\pi(X_1,\ldots,X_d) = Z_d \prod_{i=1}^d \frac1{1+(X_i-\alpha(X_{i-1}))^2},
\end{equation}
where we consider the scalar functions $\alpha(x)=x/2$ and $\alpha(x)=\sin(x)$, respectively.
Notice that the density \eqref{eq:rhononproduct} is associated with the AR(1) process
$X_i = \alpha(X_{i-1}) + Z_n$
with non Gaussian (Cauchy) increments $Z_n$. Furthermore, we observe that in this case the Jacobian in \eqref{eq:fMALA} is a 
symmetric tridiagonal matrix, which implies that the  computational cost of the fMALA proposal is of the same order  $\mathcal{O}(d)$ as the standard MALA proposal.

In Figure~\ref{fig:fignonproduct}, we compare for many timesteps the standard MALA (left pictures) and the new fMALA (right pictures),
and plot the (scaled) first order efficiency $\IE[\|X_{k+1}-X_k\|^2/d]$ as a function of the overall acceptance rates, 
using the averages over $2\times 10^4$ iterations of the algorithms.
The initial condition for both algorithms is the same and is obtained after running $10^4$ steps of the RWM algorithm to get close to the target probability measure.
Analogously to the product case studied in Figure~\ref{fig:figefficiency}, 
we observe in both cases $\alpha(x)=x/2$ and $\alpha(x)=\sin(x)$ that the first-order efficiency of fMALA converges to a non-zero limiting curve with maximum close to the value $0.704$.
In contrast, the efficiency of the standard MALA drops to zero in this scaling where the first-order efficiency 
is multiplied with $d^{1/5}$. This numerical experiment suggests that our analysis 
in the product measure setting persists in the non product measure case.
\begin{figure}[t!]
\centering
 \scalebox{0.70}{\global\def\path{#1}\input{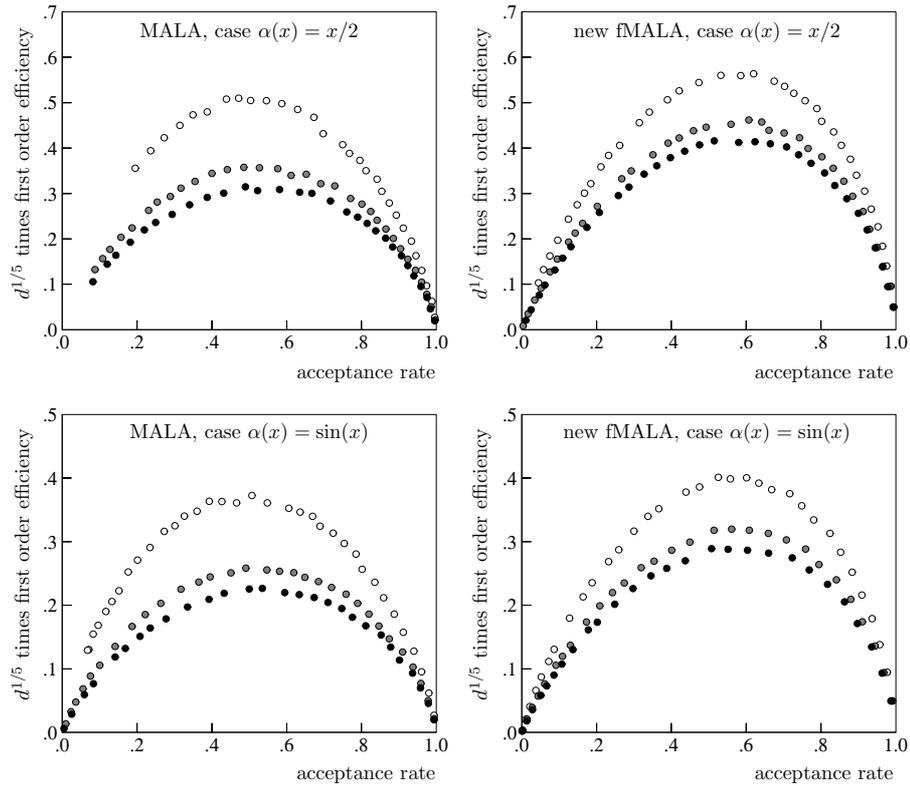}}
 \bigskip
\caption{
First-order efficiency of the new fMALA and the standard MALA as a function of the overall acceptance rates 
for the dimensions $d=100$ (white points), $d=500$ (gray points), $d=1000$ (dark points), respectively, for the non product density
\eqref{eq:rhononproduct} with $\alpha(x)=x/2$ (top pictures) and $\alpha(x)=\sin(x)$ (bottom pictures).
\label{fig:fignonproduct}}
\end{figure}

\medskip

\noindent \textbf{Acknowledgements.} 
The authors would like to thank Eric Moulines and Gabriel Stoltz for fruitful discussions, and
the anonymous referees for their useful comments that  greatly improved this paper.

\bibliographystyle{abbrv}
\bibliography{MCMC}

\appendix

\section{Proof of Theorems~3.1 and 3.2}
\label{app:scaling_proof}
We provide here the proofs of Theorems~\textrm{\ref{theo:accept_ratio_fMALA}} and \ref{theo:scaling:fMALA}  for the analysis of the optimal scaling properties of fMALA.
We use tools analogous to that of \cite{RGG97} and \cite{RR98}. Consider the generator of the jump process $\procLine^{d,\fMALA}$, defined for $\testFunction^d \in C^2_c(\rset^d)$, and $\varx \in \rset^d$ by
\[
A^{\fMALA}_{d}\testFunction^d(\varx) = d^{1/5} \expeLaw{\vary}{(\testFunction^d(\vary) - \testFunction^d(\varx) ) \alpha_d^{\fMALA}(\varx,\vary) } \eqsp,
\]
where $y$ follows the distribution defined by $q^{\fMALA}_d(\varx,\cdot)$. Also, consider the generator of the process $\{ G_t , t \geq 0\}$, solution of \eqref{eq:langevin_fMALA}, defined for $\testFunction \in C^2_c(\rset)$, and $\varx \in \rset^d$ by
\[
A^{\fMALA} \testFunction(\varx)  = (\speedMeasure(\ell) /2) ( \testFunction'(\varx_1) g(\varx_1) + \testFunction''(\varx_1)) \eqsp.
\]
We check  that the assumptions of \cite[Corollary 8.7,
Chapter 4]{ethier:kurtz:1986} are satisfied, which will imply Theorem
\ref{theo:scaling:fMALA}. These assumptions consist in showing there
exists a sequence of set $\{F_d \subset \rset^d, d \in \nset^* \}$ such that for all $T
\geq 0$:
\begin{align*}
\lim_{d \to \plusinfty} \proba{\procLine^{d,\fMALA}_s \in F_d \ , \  \forall s \in \ccint{0,T}} &= 1 \\
\lim_{d \to \plusinfty} \sup_{\varx \in F_d} \abs{A^{\fMALA}_d \testFunction(\varx) - A^{\fMALA}\testFunction(\varx)} & = 0 \eqsp,
\end{align*}
for all functions $\testFunction$ in a core of $A^{\fMALA}$, which strongly
separates points. Since $A^{\fMALA}$ is an operator on the set of
functions only depending on the first component, we restrict our
study on this class of functions, which belong to $C^{\infty}_c(\rset)$, since by \cite[Theorem 2.1,
Chapter 8]{ethier:kurtz:1986}, this set of functions  is a core for
$A^{\fMALA}$ which strongly separates points. The following lemma is
the proper result which was introduced in Section \ref{ssec:fd}. For the sequel, let $\lbrace \GaussianRV_i, i \in \nset^* \rbrace$  be a sequence of \iid~standard one-dimensional Gaussian random variables and $\xrm$ be a random variable distributed according to $\pi_1$. Also, for all $x \in \rset^d$, denote by $\vary^{\fMALA}$  the proposal of fMALA, defined by \eqref{eq:gaussian_proposal}, \eqref{eq:new_proposal} and \eqref{eq:new_proposal_2}, started at $x \in \rset^d$, with parameter $h_d$ and associated with the $d$-dimensional Gaussian random variable $\lbrace \GaussianRV_i , i=1,\cdots,d \rbrace $.
\begin{lemma}
\label{lem:taylor_expansion}
Assume Assumption \ref{hyp:log_density}.
The following Taylor expansion in $h_d^{1/2}$ holds: for all $x \in \rset^d$ and $i\in \lbrace 1,\cdots,d \rbrace$,
\begin{equation}
\label{eq:lem_taylor_expansion_form}
\log \bracket{\frac{\pi(\vary^{\fMALA}_i) q^{\fMALA}(\vary^{\fMALA}_i,\varx_i)}{\pi(\varx_i)
    q^{\fMALA}(\varx_i,\vary^{\fMALA}_i)}} = \sum_{j=5}^{10}C_j^{\fMALA}(\varx_i,\GaussianRV_i)d^{-j/10} 
+ C_{11}^{\fMALA}(\varx_i,\GaussianRV_i,h_d)  \eqsp,
  \end{equation}
  where $C_5^{\fMALA}(\varx_1,\GaussianRV_1)$ is given in Appendix~\ref{app:expression_C5}.
Furthermore, for $j=6, \cdots, 10$, $C_j^{\fMALA}(\varx_i,\GaussianRV_i)$ are polynomials in $\GaussianRV_i$ and
derivatives of $g$ at $\varx_i$ and
\begin{align}
\label{eq:lem_taylor_expansion_expe1}
\expeLaw{\GaussianRV_1}{C_j^{\fMALA}(\xrm,\GaussianRV_1)}&= 0 \text{ for $j=5, \cdots ,9$} \eqsp, \\
\label{eq:lem_taylor_expansion_expe2}
\expeLaw{\xrm}{\left(\expeLaw{\GaussianRV_1}{C_5^{\fMALA}(\xrm,\GaussianRV_1 )\left| \xrm \right.) } \right)^2 }
&= \ell^{10} (K^{\fMALA})^2 =
-2 \expeLaw{\GaussianRV_1}{C_{10}^{\fMALA}(\xrm,\GaussianRV_1)} \eqsp.
\end{align}
In addition, there exists a sequence of sets $\{F^1_d \subset \IR^d, d \in \nset^* \}$ such that $\lim_{d \to \plusinfty}
d^{1/5} \pi_d((F_d^1)^c) = 0$ and for $j=6,\cdots,10$
\begin{equation}
\label{eq:lem_taylor_lim_d_infty_c_j}
\lim_{d \to \plusinfty } d^{-3/5}  \sup_{\varx \in F_d^1}\expeLaw{\GaussianRV}{\abs{\sum_{i=2}^d C_j(\varx_i^d,\GaussianRV_i) -
    \expeLaw{\GaussianRV_i}{C_j^{\fMALA}(\xrm,\GaussianRV_i)}}  } =0 \eqsp,
    \end{equation}
    and
    \begin{equation}
\label{eq:lem_taylor_lim_d_infty_c_11}
\lim_{d \to \plusinfty } \sup_{\varx \in F_d^1}\expeLaw{\GaussianRV}{\abs{ \sum_{i=2}^d
    C_{11}(\varx_i^d,\GaussianRV_i,h_d)}} =0 \eqsp.
  \end{equation}
  Finally,
  \begin{equation}
\label{eq::lem_taylor_lim_d_infty_accept_rate}
  \lim_{d \to \plusinfty} \sup_{\varx \in
    F_d^1}\expeLaw{\GaussianRV}{\abs{ \zeta^d}} = 0 \eqsp,
    \end{equation}
    with
\begin{multline*}
 \zeta^d = \sum_{i=2}^d \log \bracket{\frac{\pi(\vary^{\fMALA}_i)
        q^{\fMALA}(\vary^{\fMALA}_i,\varx_i)}{\pi(\varx_i) q^{\fMALA}(\varx_i,\vary^{\fMALA}_i)}}\\
    - \bracket{\bracket{d^{-1/2} \sum_{i=2}^d C_5(\varx_i^d,\GaussianRV_i) }
      -\ell^{10}(K^{\fMALA})^2 /2} \eqsp.
  \end{multline*}
\end{lemma}

\begin{proof}
  The Taylor expansion was computed using the computational software
  Mathematica \cite{mathematica:10}. Then, since just odd powers of
  $\xi_i$ occur in $C_5,C_7$ and $C_9$, we deduce
  \eqref{eq:lem_taylor_expansion_expe1} for $j=5,7,9$. Furthermore by
  explicit calculation, the anti-derivative in $x_1$ of $\rme^{g(x_1)}
  \expeLaw{\xi_1}{C_j^{\fMALA}(x_1,\xi_1)}$, for $j=6,8$, and
  $\rme^{g(x_1)} \expeLaw{\xi_1}{C_5^{\fMALA}(x_1,\xi_1)^2 +2
    C_{10}^{\fMALA}(x_1,\xi_1)}$ are on the form of some polynomials
  in the derivatives of $g$ in $x_1$ times $\rme^{g(x_1)}$. Therefore,
  Assumption
  \ref{hyp:log_density}-\eqref{item:hyp_log_density_moments} implies
  \eqref{eq:lem_taylor_expansion_expe1} for $j=6,8$ and
  \eqref{eq:lem_taylor_expansion_expe2}. We now build the sequence of
  sets $F_d^1$, which satisfies the claimed properties. \\
Denote for $j=6, \cdots, 10$ and $\varx_i \in \rset$, $\tilde{C}_j^{\fMALA}(x_i) = \expeLaw{\xi_i}{C_j^{\fMALA}(x_i,\xi_i)}$ and $\Varfun_j^{\fMALA}(x_i) = \Var\left[ C_j^{\fMALA}(x_i,\xi_i) \right]$, which are bounded by a polynomial $\Prm_1$ in $x_i$ by Assumption \ref{hyp:log_density}-\eqref{item:hyp_log_density_polynomial} since $C_j^{\fMALA}(x_i,\xi_i)$ are polynomials in $\xi_i$ and the derivatives of $g$ at $x_i$. Therefore for all $k \in \nset^*$,
\begin{equation}
\label{eq:moment_bound_Cj_Vj}
\expeLaw{\xrm}{\abs{\tilde{C}_j^{\fMALA}(\xrm)}^k} + \expeLaw{\xrm}{\abs{\Varfun_j^{\fMALA}(\xrm)}^k} < ÷\plusinfty \eqsp.
\end{equation}
Consider for all $j=6,\cdots, 10$, the sequence of sets $F_{d,j}^1 \in \rset^d$ defined by $F_{d,j}^1 = F_{d,j,1}^1 \cap F_{d,j,2}^1$ where
\begin{align}
\label{eq:def_fdj1}
F_{d,j,1}^1 &=  \defSet{\varx \in \rset^d \ ;\ \abs{\sum_{i=2} ^d \tilde{C}_j^{\fMALA}(\varx_i) - \expeLaw{\xrm}{\tilde{C}_j^{\fMALA}(\xrm)}} \leq d^{23/40}} \\
\label{eq:def_fdj2}
F_{d,j,2}^1 &= \defSet{\varx \in \rset^d \ ;\ \abs{\sum_{i=2} ^d \Varfun_j^{\fMALA}(\xrm) - \expeLaw{\xrm}{\Varfun_j^{\fMALA}(\xrm)}}  \leq d^{23/20}}\eqsp.
\end{align}
Note that $\lim_{d \to \plusinfty} d^{1/5} \pi_d((F_{d,j}^1)^c)=0$ for all $j=6 \cdots 10$, is implied by  $\lim_{d \to \plusinfty} d^{1/5} \pi_d((F_{d,j,1}^1)^c) = 0$ and  $\lim_{d \to \plusinfty}d^{1/5} \pi_d((F_{d,j,2}^1)^c) = 0$.
Let $\defSet{\xrm_i, i \geq 2 }$ be a sequence of \iid~random variables with distribution $\pi_1$. By definition of $F_{d,j,1}^1$, the Markov inequality and independence, we get
\begin{align}
\nonumber
 & d^{1/5}\pi_d((F_{d,j,1}^1)^c) \leq d^{-21/10} \mathbb{E}\left[\bracket{\sum_{i=2} ^d \tilde{C}_j^{\fMALA}(\xrm_i) - \expeLaw{\xrm}{\tilde{C}_j^{\fMALA}(\xrm)}}^4\right]\\
  \nonumber
 & \leq \sum_{i_1,i_2 =2}^d \mathbb{E}\left[\bracket{ \tilde{C}_j^{\fMALA}(\xrm_{i_1}) - \expeLaw{\xrm}{\tilde{C}_j^{\fMALA}(\xrm)}}^2 \bracket{ \tilde{C}_j^{\fMALA}(\xrm_{i_2}) - \expeLaw{\xrm}{\tilde{C}_j^{\fMALA}(\xrm)}}^2 \right] \\
\label{eq:bound_fdj1}
&  \leq  d^{-1/10}\  \expeLaw{\xrm}{\bracket{ \tilde{C}_j^{\fMALA}(\xrm) -
    \expeLaw{\xrm}{\tilde{C}_j^{\fMALA}(\xrm)}}^4} \eqsp,
\end{align}
where we have used the Young inequality for the last line. On another hand, using the Chebyshev and H\"older inequality, we get
\begin{align}
\nonumber
  d^{1/5}\pi_d((F_{d,j,2}^1)^c) &\leq d^{-21/10}\mathbb{E}\left[\bracket{\sum_{i=2} ^d \Varfun_j^{\fMALA}(\xrm_i) - \expeLaw{\xrm}{\Varfun_j^{\fMALA}(\xrm)}}^2\right]\\
\label{eq:bound_fdj2}
&\leq d^{-1/10} \ \expeLaw{\xrm}{ \bracket{\Varfun_j^{\fMALA}(\xrm) - \expeLaw{\xrm}{\Varfun_j^{\fMALA}(\xrm)}}^2} \eqsp.
\end{align}
Therefore \eqref{eq:moment_bound_Cj_Vj}, \eqref{eq:bound_fdj1} and \eqref{eq:bound_fdj2} imply  that $\lim_{d \to \plusinfty} d^{1/5} \pi_d((F_{d,j}^1)^c) = 0$ for all $j=6 ,\cdots, 10$. In addition, for all $\varx\in F_{d,j}^1$, by the triangle inequality and the Cauchy-Schwarz inequality we have  for all $j=6 ,\cdots, 10$
\begin{multline*}
\expeLaw{\xi}{\abs{\sum_{i=2}^d C_j^{\fMALA}(\varx_i,\xi_i) - \expeLaw{\xi}{C_j^{\fMALA}(\xrm,\xi_i)}}} \leq  \abs{\sum_{i=2} ^d \Varfun_j^{\fMALA}(\varx_i) - \expeLaw{\xrm}{\Varfun_j^{\fMALA}(\xrm)}}^{1/2} \\
+ d^{1/2} \expeLaw{\xrm}{\Varfun_j^{\fMALA}(\xrm)}^{1/2} + \abs{\sum_{i=2}^d \tilde{C}_j^{\fMALA}(\varx_i) - \expeLaw{\xi_i}{C_j^{\fMALA}(\xrm,\xi_i)}}  \eqsp.
\end{multline*}
Therefore by this inequality, \eqref{eq:def_fdj1} and \eqref{eq:def_fdj2}, there exists a constant $M_1$ such that
\[
d^{3/5} \sup_{\varx \in F_{d,j}^1} \expeLaw{\xi}{\abs{\sum_{i=2}^d C_j^{\fMALA}(\varx_i,\xi_i) - \expeLaw{\xi}{C_j^{\fMALA}(\xrm,\xi_i)}}} \leq d^{-1/40} M_1 \eqsp,
\]
and \eqref{eq:lem_taylor_lim_d_infty_c_j} follows. It remains to show \eqref{eq:lem_taylor_lim_d_infty_c_11}. By definition, $C_{11}$ is the remainder in the eleventh order  expansion in $\sigma_d := \sqrt{h_d}$  given by \eqref{eq:lem_taylor_expansion_form} of the function $\Theta$ defined by $\Theta(\varx_i,\xi_i,\sigma_d) = \log(\pi_1(y_i^{\fMALA} ) q^{\fMALA}_1(y_i^{\fMALA},x_i)) - \log(\pi_1(x_i) q_1^{\fMALA}(x_i,y_i^{\fMALA}))$. Therefore, by the mean-value form of the remainder, there exists $u_d \in \ccint{0,\sigma_d}$ such that
\[
C_{11}(x_i,\xi_i,h_d) = (\sigma_d^{11}/(11!)) \frac{\partial ^{11} \Theta}{\partial \sigma_d^{11}}(x_i,\xi_i,u_d) \eqsp.
\]
By Assumption \ref{hyp:log_density}-\eqref{item:hyp_log_density_lipschitz_Cten} which implies that $g''$ is bounded, and Assumption \ref{hyp:log_density}-\eqref{item:hyp_log_density_polynomial}, for all $u_d \in \ccint{0,\sigma_d}$, the eleventh derivative of $\Theta$ with respect to $\sigma_d$, taken in $(x_i,\xi_i,u_d)$,  can be bounded by a positive polynomial in $(\varx_i,\xi_i)$ on the form  $\Prm_2(\varx_i) \Prm_3(\xi_i)$. Hence, there exists a constant $M_2$ such that
\begin{equation}
\label{eq:lem_taylor_expansion_c_11_bound1}
\expeLaw{\xi_i} {\abs{C_{11}(\varx_i,\xi_i,h_d)}} \leq M_2 \ d^{-11/10} \ \Prm_2(\varx_i) \eqsp.
\end{equation}
And if we define
\[
F_{d,11}^1 = \defSet{x \in \rset^d \ ; \ \abs{\sum_{i=2}^d \Prm_2(\varx_i) - \expeLaw{\xrm}{\Prm_2(\xrm)}} \leq d} \eqsp,
\]
then we have by the Chebychev inequality, this definition  and \eqref{eq:lem_taylor_expansion_c_11_bound1}
\begin{align*}
  d^{1/5} \pi_d((F_{d,11}^1)^c) & \leq \Var\left[ \Prm_2(\xrm) \right] d^{-4/5} \\
  \sup_{\varx \in F_{d,11}^1 } \sum_{i=2}^d   \expeLaw{\xi_i}{\abs{C_{11}(\varx_i,\xi_i,h_d)}}
&\leq M_2   (\expeLaw{\xrm}{\Prm_2(\varx)} +1)d^{-1/10} \eqsp.
\end{align*}
These results, combined with Assumption \ref{hyp:log_density}-\eqref{item:hyp_log_density_moments}, imply $\lim_{d \to \plusinfty} d^{1/5} \pi_d((F_{d,11}^1)^c) =0$ and \eqref{eq:lem_taylor_lim_d_infty_c_11}. Finally, $F_{d}^1 = \bigcap_{j=6}^{11} F_{d,j}^1$ satisfies the claimed properties of the Lemma, and \eqref{eq::lem_taylor_lim_d_infty_accept_rate} directly follows from all the previous results.
\end{proof}
To isolate the first component of the process $\procLine^{d, \fMALA}$, we consider the modified generators defined for $\testFunction \in C^2_c(\rset ^d)$ and $\varx \in \rset^d$ by
\[
\tilde{A}_{d}^{\fMALA}\testFunction(\varx) = d^{1/5} \expeLaw{\vary^{\fMALA}}{(\testFunction(\vary^{\fMALA})-\testFunction(\varx))\alpha_{-1,d}^{\fMALA}(\varx,\vary^{\fMALA})}
\]
where for all $x,y \in \rset^d$,
\[
\alpha_{-1,d}^{\fMALA} (\varx,\vary) = \prod_{i=2}^d \frac{\pi_1(\vary_i)q_{1,\fMALA}(\vary_i,\varx_i)}{\pi_1(\varx_i)q_{1,\fMALA}(\varx_i,\vary_i)} \eqsp.
\]
The next lemma shows that we can approximate $A_d^{\fMALA}$ by $\tilde{A}_{d}^{\fMALA}$, and thus, in essence, the first component becomes ``asymptotically independent'' from the others.

\begin{theorem}
\label{theo:approxAd}
There exists a sequence of sets $\{ F_d^2 \subset \rset^d, d \in \nset^* \}$ such that $\lim_{d \to \plusinfty} d^{1/5} \pi_d((F_d^2)^c) =0$ and  for all $\testFunction \in C_c^{\infty}(\rset)$ (seen as function of $\rset^d$ for all $d$ which only depends on the first component):
\begin{equation*}
\lim_{d \to \plusinfty} \sup_{\varx \in F^{2}_d } \abs{A_d^{\fMALA}\testFunction(\varx) - \tilde{A}_{d}^{\fMALA}\testFunction(\varx) } = 0 \eqsp.
\end{equation*}
In addition,
\begin{equation}
\label{eq:theo_approx_AdTwo}
\lim_{d \to \plusinfty} \sup_{\varx \in F_d^2} d^{1/5} \expeLaw{\vary^{\fMALA}}{\abs{\alpha_d^{\fMALA}(\varx,\vary^{\fMALA}) - \alpha_{-1,d}^{\fMALA}(\varx,\vary^{\fMALA})}} = 0 \eqsp.
\end{equation}
\end{theorem}

\begin{proof}
Using that $\testFunction$ is bounded and the Jensen inequality, there exists a constant $M_1$ such that
\[
\abs{A_d^{\fMALA}\testFunction(\varx) - \tilde{A}_{d}^{\fMALA}\testFunction(\varx) } \leq M_1 d^{1/5}\expeLaw{\vary^{\fMALA}}{\abs{\alpha_d^{\fMALA}(\varx,\vary^{\fMALA}) - \alpha_{-1,d}^{\fMALA}(\varx,\vary^{\fMALA})}}.
\]
Thus it suffices to show \eqref{eq:theo_approx_AdTwo}.
Set $\sigma_d = \sqrt{h_d}$.
Since $t \mapsto 1 \wedge \exp(t)$ is $1$-Lipschtz on $\rset$ and, by definition we have
\begin{equation}
\label{eq:proof_theo_approxAd}
 d^{1/5}\expeLaw{\vary^{\fMALA}}{\abs{\alpha_d^{\fMALA}(\varx,\vary^{\fMALA}) - \alpha_{-1,d}^{\fMALA}(\varx,\vary^{\fMALA})}}  \leq  d^{1/5} \expeLaw{\xi_1}{\abs{\Theta(\varx_1,\xi_1,\sigma_d)}} \eqsp,
\end{equation}
where $\Theta(\varx_1,\xi_1,\sigma_d) = \log(\pi_1(\vary_1^{\fMALA})q_1^{\fMALA}(\vary_1^{\fMALA},\varx_1)) - \log(\pi_1(\varx_1)q_1^{\fMALA}(\varx_1,\vary^{\fMALA}_1))$. By a fifth order Taylor expansion of $\Theta$ in $\sigma_d$, and since by \eqref{eq:lem_taylor_expansion_form} $\partial^j \Theta(\varx_1,\xi_1,0)/(\partial \sigma_d^j) = 0$ for $j=0 \cdots 4$, we have
\[
\Theta(\varx_1,\xi_1,\sigma_d) = \frac{\partial^5 \Theta}{\partial \sigma_d^5}(\varx_1,\xi_1,u_d) (\sigma_d^5/5!) \eqsp,
\]
for some $u_d \in \ccint{0,\sigma_d}$.
Using Assumption \ref{hyp:log_density}-\eqref{item:hyp_log_density_lipschitz_Cten}-\eqref{item:hyp_log_density_polynomial}, and an explicit expression of $\partial^j \Theta(\varx_1,\xi_1,u_d)/(\partial \sigma_d^j)$, there exists two positive polynomials $\Prm_1$ and $\Prm_2$ such that
\[
\abs{\Theta(\varx_1,\xi_1,\sigma_d)} \leq (\sigma_d^5/5!) \Prm_1(\varx_1) \Prm_2(\xi_1) \eqsp.
\]
Plugging this result in \eqref{eq:proof_theo_approxAd} and since $\sigma_d^5 = \ell^{5/2} d^{-1/2}$, we get
\[
 d^{1/5}\expeLaw{\vary^{\fMALA}}{\abs{\alpha_d^{\fMALA}(\varx,\vary^{\fMALA}) - \alpha_{-1,d}^{\fMALA}(\varx,\vary^{\fMALA})}}  \leq  \ell^{5/2} d^{-3/10} \Prm_1(\varx_1) \eqsp.
\]
Setting $F^2_d = \{ \varx \in \rset^d \ ; \ \Prm_1(\varx_1) \leq d^{1/10} \}$, we have
\[
\sup_{\varx \in F^2_d}d^{1/5}\expeLaw{\vary^{\fMALA}}{\abs{\alpha_d^{\fMALA}(\varx,\vary^{\fMALA}) - \alpha_{-1,d}^{\fMALA}(\varx,\vary^{\fMALA})}}  \leq \ell^{5/2} d^{-1/5} \eqsp,
\]
and \eqref{eq:theo_approx_AdTwo} follows. Finally, $F^2_d$ satisfied $\lim_{d \to \plusinfty}d^{1/5} \pi_d((F_d^2)^c) = 0$ since by the Markov inequality
\[
d^{1/5} \pi_d((F_d^2)^c) \leq d^{-1/10} \expeLaw{\xrm}{\Prm_1(\xrm)^3} \eqsp,
\]
where $\expeLaw{\xrm}{\Prm_1(\xrm)^3}$ is finite by
Assumption \ref{hyp:log_density}-\eqref{item:hyp_log_density_moments}.
\end{proof}
\begin{lemma}
\label{lem:approx_A_langevin}
For all $\testFunction \in C^\infty_c(\rset)$,
\[
\lim_{d \to \plusinfty} \sup_{\varx_1 \in \rset} \abs{d^{1/5} \expeLaw{y_1^{\fMALA}}{\testFunction(y_1^{\fMALA})-\testFunction(x_1)} - (\ell^2/2)(\testFunction'(\varx_1)f(\varx_1) + \testFunction''(\varx_1)) } = 0 \eqsp.
\]
\end{lemma}
\begin{proof}
Consider $\sigma_d = \sqrt{h_d}$ and $W(\varx_1, \xi_1 , \sigma_d) =
\testFunction(y^{\fMALA}_1)$. Note that $W(\varx_1, \xi_1 , 0) = \testFunction(\varx_1)$. Then using that $\testFunction \in
C^\infty_c(\rset)$, a third order Taylor expansion of this function in $\sigma_d$ implies
there exists $u_d \in \ccint{0,h_d}$ and $M_1 \geq 0$ such that
\begin{multline*}
\expeLaw{\xi_1}{W(\varx_1,\xi_1,\sigma_d) - \testFunction(\varx_1) } = (\ell^2 d^{-1/5}
/2)(\testFunction'(\varx_1)f(\varx_1) + \testFunction''(\varx_1))+M_1 d^{-3/10}\\ + \frac{\partial^3 W}{\partial \sigma_d^3}(x_1,\xi_1,u_d)\sigma_d^3 \eqsp.
\end{multline*}
Moreover since $\testFunction \in C^\infty_c(\rset)$, the third partial derivative of $W$ in
$\sigma_d$ are bounded for all $\varx_1$, $\xi_1$ and $\sigma_d$. Therefore there exists
$M_2 \geq 0$ such that  for all $\varx_1 \in \rset$,
\[
\abs{d^{1/5} \expeLaw{y_1^{\fMALA}}{\testFunction(y_1^{\fMALA})-\testFunction(x_1)} - (\ell^2/2)(\testFunction'(\varx_1)f(\varx_1) + \testFunction''(\varx_1)) } \leq M_2 \ell^{3/2} d^{-1/10} \eqsp,
\]
which concludes the proof.
\end{proof}

As in \cite{RR98}, we prove a uniform central limit theorem for the sequence of random variables defined for $i \geq 2$ and $\varx_i \in \rset$ by
$C_5^{\fMALA}(\varx_i,\xi_i)$. Define now for $d \geq 2$ and  $\varx \in \rset^d$,
\[
\Wbar_d(  \varx) = n^{-1/2} \sum_{i=2}^d C_5^{\fMALA}(\varx_i,\xi_i) \eqsp,
 \]
and the characteristic function of $\Wbar_d$ for $t \in \rset$ by
\[
\varphi_d(\varx,t) = \mathbb{E}[\rme^{\rmi t \Wbar_d(\varx)}] \eqsp.
\]
Finally define the characteristic function of the zero-mean Gaussian distribution with
standard deviation $\ell^5 K^{\fMALA}$, given in Lemma \ref{lem:taylor_expansion}, by: for $t \in \rset$,
\[
\varphi(t) = \rme^{-(\ell^{5} K^{\fMALA}t)^2 /2} \eqsp.
\]
\begin{lemma}
\label{lem:unif_TCL}
 There exists a sequence of set $\{F_d^3 \subset \rset^d , d \in \nset^* \}$, satisfying $\lim_{d \to \plusinfty} d^{1/5}\pi_d((F_d^3)^c) = 0$ and we have the following properties:
 \begin{enumerate}[(i)]
  \item  \label{item:unif_TCL1}
  for all $t \in \rset$, $\lim_{d \to \plusinfty} \sup_{\varx \in F_d^3} \abs{\varphi_d(x,t) - \varphi(t)} = 0$,
  \item \label{item:unif_TCL2}
  for all bounded continuous function $\brm : \rset \to \rset$,
  \[
    \hspace{-1cm}
   \lim_{d \to \plusinfty} \sup_{\varx \in F^3_d} \abs{\mathbb{E}\left[ \brm \left( \Wbar_d(\varx) \right) \right] - (2 \uppi \ell^{10} (K^{\fMALA})^2)^{-1/2} \int_{\rset } \brm(u) \rme^{-u^2/(2\ell^{10} (K^{\fMALA})^2)} \rmd u} = 0 \eqsp.
  \]
\end{enumerate}
In particular, we have
\[
    \lim_{d \to \plusinfty} \sup_{\varx \in F^3_d} \abs{\mathbb{E}\left[ 1 \wedge \rme^{\Wbar_d(\varx) -\ell^{10} (K^{\fMALA})^2/2 } \right] - 2 \Phi(\ell^5 K^{\fMALA}/2)} =0 \eqsp.
\]
\end{lemma}

\begin{proof}
 We first define for all $d \geq 1$, $F_d^3 = F_{d,1}^3 \cap F_{d,2}^3$ where
 \begin{align}
 \label{eq:defFd31}
 F_{d,1}^3 & = \bigcap_{j=2,4} \defSet{\varx \in \rset^d \ ; \abs{\ d^{-1} \sum_{i=2}^d \expeLaw{\xi_i}{ C_5^{\fMALA}(\varx_i,\xi_i)^j}  - \expeLaw{\xi_1}{C_5^{\fMALA}(\xrm_1,\xi_1 )^j}   } \leq d^{-1/4}} \eqsp,  \\
 \label{eq:defFd32}
 F_{d,2}^3 & = \defSet{\varx \in \rset^d \ ; \ \expeLaw{\xi_i}{C_5^{\fMALA}(\varx_i,\xi_i)^2} \leq d^{3/4} \, \  \forall i \in \{ 2, \cdots , d \} } \eqsp.
 \end{align}
It follows from \eqref{eq:moment_bound_Cj_Vj}, and the Chebychev and Markov inequalities that there exists a constant $M$ such that $\pi_d((F^3_{d,1})^c)+\pi_d((F^3_{d,2})^c) \leq M d^{-1/2}$.
Therefore $\lim_{d\to \plusinfty}d^{1/5} \pi_d((F^3_d)^c) = 0$. \\
\eqref{item:unif_TCL1}. Let $t \in \rset$ and $\varx \in F_d^3$ and denote
\[
 \Varfun(\varx_i) = \Var[C_5^{\fMALA}(x_i,\xi_i)] = \expeLaw{\xi_i}{C_5^{\fMALA}(\varx_i,\xi_i)^2} \eqsp,
\]
where the second equality follows from Lemma \ref{lem:taylor_expansion}.
By the triangle inequality
\begin{multline}
\label{eq:unif_TCL_first_ineq_carac_fun}
\abs{\varphi_d(\varx,t) - \varphi(t)} \leq \abs{\varphi_d(\varx,t) - \prod_{i=2}^d
  \bracket{1- \frac{\ell^{10}\Varfun(\varx_i)t^2}{2d} }} \\+ \abs{\prod_{i=2}^d
  \bracket{1- \frac{\ell^{10}\Varfun(\varx_i)t^2}{2d} } - \rme^{-\ell^{10} (K^{\fMALA})^2
    t^2 /2} } \eqsp.
\end{multline}
We bound  the two terms of the right hand side separately.
Note that by independence for all $d$, $\varphi_d(\varx,t) = \prod_{i=2}^d \varphi_1(\varx_i,t/\sqrt{d})$.
Since $\varx \in F_d^3$, by \eqref{eq:defFd32}, for $d$ large enough $\ell^{10} \Varfun(\varx_i)t^2/(2d) \leq 1 $ for all $i \in \{2,\cdots,d \}$. Thus, by   \cite[Eq. 26.5]{billingsley:1995}, we have for such large $d$, all $i \in \{2,\cdots,d\}$ and all $\delta >0$:
\begin{align*}
&\abs{\varphi_1(\varx_i,t/\sqrt{d}) -\bracket{1- \frac{\ell^{10}\Varfun(\varx_i)t^2}{2d} }}\\
 &\leq \expeLaw{\xi_i}{\bracket{\frac{\abs{t}^3\ell^{15}} {6d^{3/2}} \abs{C_5^{\fMALA}(\varx_i,\xi_i)}^3 }\wedge \bracket{\frac{t^2 \ell^{10}}{d} C_5^{\fMALA}(\varx_i,\xi_i)^2} }  \\
& \leq \expeLaw{\xi_i}{\frac{\abs{t}^3\ell^{15}} {6d^{3/2}} \abs{C_5^{\fMALA}(\varx_i,\xi_i)}^3 \1_{\{\abs{C_5^{\fMALA}(\varx_i,\xi_i)} \leq \delta d^{1/2}\}} } \\
&\qquad+\expeLaw{\xi_i}{\frac{t^2 \ell^{10}}{d} C_5^{\fMALA}(\varx_i,\xi_i)^2  \1_{\{ \abs{C_5^{\fMALA}(\varx_i,\xi_i)} > \delta d^{1/2} \}} } \\
& \leq \frac{\delta \abs{t}^3 \ell^{15}}{6d} \expeLaw{\xi_i}{  C_5^{\fMALA}(\varx_i,\xi_i)^2 } + \frac{\ell^{10}t^2}{\delta^2 d^2}  \expeLaw{\xi_i}{  C_5^{\fMALA}(\varx_i,\xi_i)^4} \eqsp,
\end{align*}
In addition, by \cite[Lemma 1, Section 27]{billingsley:1995} and using this result we get:
\begin{align*}
\abs{\varphi_d(\varx,t) - \prod_{i=2}^d \bracket{1- \frac{\ell^{10}\Varfun(\varx_i)t^2}{2d} } } & \leq
\sum_{i=2}^d \frac{\delta \abs{t}^3 \ell^{15}}{6d} \expeLaw{\xi_i}{  C_5^{\fMALA}(\varx_i,\xi_i)^2 }\\
 &\qquad+ \frac{\ell^{10}t^2}{\delta^2 d^2}  \expeLaw{\xi_i}{  C_5^{\fMALA}(\varx_i,\xi_i)^4}  \\
&\leq \bracket{\expeLaw{\xi_1}{C_5^{\fMALA}(\xrm_1,\xi_1 )^2}  + d^{-1/4}} \ell^{15} \delta \abs{t}^3/6 \\
&+  \bracket{\expeLaw{\xi_1}{C_5^{\fMALA}(\xrm_1,\xi_1 )^4}   + d^{-1/4}} \ell^{10}t^2/(\delta^2 d) \eqsp,
\end{align*}
where the last inequality follows from $x \in F_d^3$ and \eqref{eq:defFd31}
Let now $\epsilon >0$, and choose $\delta$ small enough such that the fist term is smaller than $\epsilon/2$. Then there exists $d_0 \in \nset^*$ such that for all $d \geq d_0$, the second term is smaller than $\epsilon/2$ as well. Therefore, for $d \geq d_0$ we get
\[
\sup_{x \in F^3_d}\abs{\varphi_d(\varx,t) - \prod_{i=2}^d \bracket{1- \frac{\ell^{10}\Varfun(\varx_i)t^2}{2d} } } \leq \epsilon \eqsp.
\]
Consider now the second term of \eqref{eq:unif_TCL_first_ineq_carac_fun}, by the triangle inequality,
\begin{multline}
\label{eq:second_term}
\abs{\prod_{i=2}^d \bracket{1- \frac{\ell^{10}\Varfun(\varx_i)t^2}{2d} } - \rme^{-\ell^{10} (K^{\fMALA})^2 t^2 /2} }\\ \leq \abs{\prod_{i=2}^d \bracket{1- \frac{\ell^{10}\Varfun(\varx_i)t^2}{2d} } - \prod_{i=2}^d \rme^{-\ell^{10}\Varfun(x_i)t^2/(2d)} } \\
+ \abs{ \prod_{i=2}^d \rme^{-\ell^{10}\Varfun(x_i)t^2/(2d)} - \rme^{-\ell^{10} (K^{\fMALA})^2 t^2 /2}}  \eqsp.
\end{multline}
We deal with the two terms separatly.
First since for all $\varx_i$, $\Varfun(\varx_i) \geq 0$, we have
\[
\abs{1-\Varfun(\varx_i)\ell^{10}t^2/(2d) - \rme^{-\Varfun(\varx_i)\ell^{10}t^2/(2d) }} \leq \Varfun(\varx_i)^2\ell^{20}t^4/(8d^2)\eqsp.
\]
Using this result, \cite[Lemma 1, Section 27]{billingsley:1995} and the Cauchy-Schwarz inequality, it follows:
\begin{align}
\nonumber
&\abs{\prod_{i=2}^d \bracket{1- \frac{\ell^{10}\Varfun(\varx_i)t^2}{2d} } - \prod_{i=2}^d \rme^{-\ell^{10}\Varfun(x_i)t^2/(2d)} }\\
\nonumber
& \leq \sum_{i=2}^d
\abs{1-\Varfun(\varx_i)\ell^{10}t^2/(2d) - \rme^{-\Varfun(\varx_i)\ell^{10}t^2/(2d) }} \\
\label{eq:unif_TCL_first_term_second_term}
& \leq \sum_{i=2}^d \Varfun(\varx_i)^2\ell^{20}t^4/(8d^2) 
 \leq \bracket{\expeLaw{\xi_1}{C_5^{\fMALA}(\xrm_1,\xi_1 )^4}   + d^{-1/4}} \ell^{20}t^{4}/(8d) \eqsp,
\end{align}
where the last inequality is implied by \eqref{eq:defFd31}. Finally since on $\rset_-$, $u
\mapsto \rme^u$ is $1$-Lipschitz and using \eqref{eq:defFd31}, we get
\begin{align}
\nonumber
\abs{ \prod_{i=2}^d \rme^{-\ell^{10}\Varfun(x_i)t^2/(2d)}-  \rme^{-\ell^{10} (K^{\fMALA})^2 t^2 /2}} & \leq (t^2\ell^{10}/2)\abs{\sum_{i=2}^dd^{-1}\Varfun(\varx_i) - (K^{\fMALA})^2} \\
\label{eq:unif_TCL_second_term_second_term}
& \leq t^2\ell^{10}d^{-1/4}/2 \eqsp.
\end{align}
Therefore, combining \eqref{eq:unif_TCL_first_term_second_term} and \eqref{eq:unif_TCL_second_term_second_term} in \eqref{eq:second_term}, we get:
\[
\lim_{d\to \plusinfty} \sup_{x \in F^3_d} \abs{\prod_{i=2}^d \bracket{1- \frac{\ell^{10}\Varfun(\varx_i)t^2}{2d} } - \rme^{-\ell^{10} (K^{\fMALA})^2 t^2 /2} } = 0 \eqsp,
\]
which concludes the proof of \eqref{item:unif_TCL1}.\\
 \eqref{item:unif_TCL2} 
Let $\brm : \rset \to \rset$ be a  bounded continuous function.
Consider the sequence $\{x^d \ ,  \ d \in \nset^* \}$ of elements of $F^3_d$ which satisfies for all $d \in \nset^*$,
\begin{multline}
\label{eq:unif_TCL2b}
  \sup_{\vary \in F^3_d} \abs{\mathbb{E}\left[ \brm \left( \Wbar_d(\vary) \right) \right] - (2 \uppi \ell^{10} (K^{\fMALA})^2)^{-1/2} \int_{\rset } \brm(u) \rme^{-u^2/(2\ell^{10} (K^{\fMALA})^2)} \rmd u}  \\
\leq  \abs{\mathbb{E}\left[ \brm \left( \Wbar_d(\varx^d) \right) \right] - (2 \uppi \ell^{10} (K^{\fMALA})^2)^{-1/2} \int_{\rset } \brm(u) \rme^{-u^2/(2\ell^{10} (K^{\fMALA})^2)} \rmd u} + d^{-1} \eqsp.
\end{multline}
Then using \eqref{item:unif_TCL1} and Levy's continuity theorem, we get
  \[
   \lim_{d \to \plusinfty}  \abs{\mathbb{E}\left[ \brm \left( \Wbar_d(\varx^d) \right) \right] - (2 \uppi \ell^{10} (K^{\fMALA})^2)^{-1/2} \int_{\rset } \brm(u) \rme^{-u^2/(2\ell^{10} (K^{\fMALA})^2)} \rmd u} = 0 \eqsp.
  \]
This limit and \eqref{eq:unif_TCL2b} conclude the proof.
\end{proof}

\begin{proof}[proof of Theorem \ref{theo:accept_ratio_fMALA}]
The theorem follows from Lemma \ref{lem:taylor_expansion}, \eqref{eq:theo_approx_AdTwo} in Theorem \ref{theo:approxAd} and the last statement in Lemma \ref{lem:unif_TCL}.
\end{proof}

\begin{proof}[proof of Theorem \ref{theo:scaling:fMALA}]
Consider $F_d = \bigcap_{j=1,2,3} F_d^j$, where the sets $F_d^j$ are given resp. in Lemma \ref{lem:taylor_expansion} Theorem \ref{theo:approxAd} and Lemma \ref{lem:unif_TCL}. We then obtain $\lim_{d \to \plusinfty} d^{-1/5} \pi_d((F_d)^c) = 0$ and by the union bound, for all $T \geq 0$,
\[
\lim_{d \to \plusinfty} \proba{\procLine^{d,\fMALA}_s \in F_d \ , \  \forall s \in \ccint{0,T}} = 1  \eqsp.
\]
Furthermore, combining the former results with  Lemma \ref{lem:approx_A_langevin}, we have  for all $\testFunction \in C^\infty_c(\rset)$ (seen as a function of the first component):
\[
\lim_{d \to \plusinfty} \sup_{\varx \in F_d} \abs{A^{\fMALA}_d \testFunction(\varx) - A^{\fMALA}\testFunction(\varx)} = 0 \eqsp.
\]
Then, the weak convergence follows from \cite[Corollary 8.7, Chapter 4]{ethier:kurtz:1986}.
\end{proof}


\section{Postponed proofs}
\label{sec:app_tech_proof}
\subsection{Proof of Lemma~\ref{propo:convergence_gaussian_proposal}}
\label{subsec:lemme_convergence_gaussian_proposal}
    By Assumption \ref{hyp:Gaussian_Markov_kernel}-\ref{hyp:density_convergence}, $\pi$ and $q$ are positive and continuous. It follows from  \cite[Lemma
    1.2]{mengersen:tweedie:1996} that  $P$ is $\Leb^d$-irreducible aperiodic, where $\Leb^d$ is the
    Lebesgue measure on $\IR^d$. In addition, all compact set $\smallSet$ such that $\Leb^d(\smallSet)
    >0$ are small for $P$. Now by \cite[Theorem 15.0.1]{meyn:tweedie:2009}, we just need
    to check the drift condition \eqref{eq:drift_condition}. But by a simple calculation, using $\alpha(x,y) \leq 1$ for all $x,y \in \rset^d$, and the Cauchy-Schwarz inequality, we get
\begin{align*}
  PV(\varx) 
  & \leq 1 +  \norm{\varx}^2
  + (\norm{\mu(\varx)}^2-\norm{\varx}^2) \int_{\IR^d} \alpha(\varx,\vary) q(\varx,\vary) \rmd \vary \\
  & + (2 \uppi)^{-d/2} (2\norm{\mu(\varx)} \norm{S(\varx)} + \norm{S(\varx)}^2)\int_{\IR^d} \max(\norm{
    \GaussianRV}^2 , 1)
  \rme^{-\norm{\GaussianRV}^2/2} \rmd \GaussianRV \eqsp.
  \end{align*}
  By \eqref{eq:hyp:gaussian_prop_mu}, $ \limsup_{\norm{x} \to
    \plusinfty}(2\norm{\mu(\varx)} \norm{S(\varx)} +
  \norm{S(\varx)}^2)\norm{x}^{-2} = 0  $. Therefore, using again
  the first inequality of \eqref{eq:hyp:gaussian_prop_mu} and  Assumption \ref{hyp:ratio_convergence_bO}:
\begin{equation*}
\limsup_{\norm{\varx} \to \plusinfty} PV(\varx) /V(\varx) \leq 1-(1-\tau^2) \liminf _{\norm{\varx} \to \plusinfty} \int_{\IR^d} \alpha(\varx,\vary) q(\varx,\vary) \rmd
\vary < 1 \eqsp.
\end{equation*}
 This concludes the proof  of Lemma~\ref{propo:convergence_gaussian_proposal}.\qed
\subsection{Proof of Theorem~\ref{theo:non_geo_erg_Gaussian_MH}}
\label{subsec:proof_theo_non_geo}
We prove this result by contradiction. The strategy of the proof is the following: first, under our
assumptions, most of the proposed moves by the algorithm has a norm which is
greater than the current point. However,  if $P$ is geometrically ergodic, then it implies a  upper bound on the rejection
probability of the algorithm by some constant strictly smaller than $1$. But combining these facts, we can exhibit a sequence of point
$\{ x_n , n \in \nset \}$, such that $\lim_{n \to \plusinfty} \pi(x_n) =
\plusinfty$. Since we assume that $\pi$ is bounded, we have our contradiction.

If $P$ is geometrically ergodic, then by
\cite[Theorem 5.1]{roberts:tweedie:1996}, there exists $\eta >0$ such that for almost every $x \in
\rset^d$,
\begin{equation}
  \label{eq:bound_rejection_proba}
\int_{\rset^d} \alpha(x,y) q(x,y) \rmd y \geq  \eta \eqsp,
\end{equation}
and let $M \geq 0$ such that
\begin{equation}
\label{eq:def_M_non_geo}
\mathbb{P}\left[ \norm{\xi} \geq M \right] \leq \eta /2 \eqsp,
\end{equation}
where $\xi$ is a standard $d$-dimensional Gaussian random variable. 
By \eqref{eq:assum_theo_non_geo}, there
exist $\Repsilon , \delta > 0$ such that
\begin{align}
  \label{def_R_epsilon_non_geo1}
\inf_{\{ \norm{x} \geq \Repsilon \}} \norm{S(x)^{-1} \mu(x)}\norm{x}^{-1} &\geq
\epsilon^{-1} + \delta \ \\
  \label{def_R_epsilon_non_geo2}
 \inf_{\{ \norm{x} \geq \Repsilon \}} \inf_{\norm{z}
    = 1} \norm{S(x)z} &\geq \epsilon (1+ \delta\epsilon/2)^{-1} \eqsp.
\end{align}
Note that we can assume $\Repsilon$ is large enough so that
\begin{equation}
  \label{eq:condition_R_epsi}
\epsilon \delta R _\epsilon/2
\geq M   \eqsp.
\end{equation}
Now define for $x \in \rset^d$, $\norm{x} \geq \Repsilon$
\begin{equation}
  \label{eq:defBx}
B(x) = \left\lbrace y \in \rset^d \ |  \   \norm{S(x)^{-1} (y - \mu(x) )} \leq M
\right\rbrace \eqsp.
\end{equation}
Note if $y \in B(x)$, we have by definition and the triangle inequality $\norm{S(x)^{-1}y } \geq \norm{S(x)^{-1}
  \mu(x)} -M$. Therefore by   \eqref{def_R_epsilon_non_geo1}-\eqref{def_R_epsilon_non_geo2} and \eqref{eq:condition_R_epsi}
\begin{multline}
  \label{eq:norm_propo_greater}
\norm{y} = \norm{S(x) S(x)^{-1} y } \geq \epsilon (1+ \delta\epsilon/2)^{-1} \norm{S(x)^{-1} y} \\
\geq
\epsilon (1+ \delta\epsilon/2)^{-1} \defEns{(\epsilon^{-1}+\delta)\norm{x} -M} \geq \norm{x} \eqsp.
  \end{multline}
  We then show that this inequality implies
\begin{equation}
  \label{eq:ratio_q_zeros}
  \liminf_{\norm{x} \to \plusinfty} \inf_{y \in
  B(x)} \frac{q(y,x)}{q(x,y)} = 0 \eqsp.
\end{equation}
Let $x \in \rset^d$, $\norm{x}\geq R_{\epsilon}$, $y \in B(x)$. First, it is straightforward by \eqref{eq:defBx}, that $\abs{S(x)}q(x,y)$ is uniformly
bounded away from $0$, and it suffices to consider $\abs{S(x)}q(y,x)$. By
\eqref{def_R_epsilon_non_geo2}-\eqref{eq:norm_propo_greater}, we have $\norm{y} \geq \Repsilon$ and for all $z \in
\rset^d$, $\norm{S(y)z} \geq \epsilon (1+ \delta\epsilon/2)^{-1} \norm{z}$, which implies for all $z \in \rset^d$,
$\epsilon^{-1}(1+ \delta\epsilon/2)\norm{z} \geq \norm{S(y)^{-1} z}$. By this inequality and  \eqref{def_R_epsilon_non_geo1}, we have 
\begin{multline}
  \label{eq:bounded_away}
\abs{\norm{S(y)^{-1}\mu(y)} - \norm{S(y)^{-1} x}} \geq \norm{S(y)^{-1}\mu(y)} - \norm{S(y)^{-1} x} \\
\geq  ( \epsilon^{-1} + \delta) \norm{y} - \epsilon^{-1}(1+ \delta\epsilon/2)\norm{x} \geq
(\delta/2 ) \norm{y} \eqsp,
\end{multline}
where the last inequality follows from \eqref{eq:norm_propo_greater}.
Using this result, the triangle 
inequality,   \eqref{eq:bounded_away}-\eqref{def_R_epsilon_non_geo2} and \eqref{eq:norm_propo_greater}, we get
\begin{align*}
q(y,x) &= (2\uppi)^{-d/2}\exp \defEns{-(1/2) \norm{S(y)^{-1} (x-\mu(y))}^2 - \log( \abs{S(y) } ) } \\
& \leq  (2\uppi)^{-d/2}\exp \defEns{-(1/2)\bracket{\norm{S(y)^{-1}\mu(y)} - \norm{S(y)^{-1} x}}^2 - \log(
  \abs{S(y) } ) } \\
& \leq   (2\uppi)^{-d/2}\exp \defEns{-(\delta^2/8) \norm{y}^2 - \log(
  \abs{S(y) } ) } \\
& \leq   (2\uppi)^{-d/2}\exp \defEns{-(\delta^2/8) \norm{x}^2 - d \log(\epsilon (1+ \delta\epsilon/2)^{-1} ) } \eqsp.
  \end{align*}
  Using this inequality and \eqref{eq:assum_theo_non_geo2} imply $\lim_{\norm{x} \to
    \plusinfty} \inf_{y \in B(x)} \abs{S(x)}q(y,x) =0$ and then
  \eqref{eq:ratio_q_zeros}. Therefore there exists $\Rqtransition \geq 0$ such that for all $x \in
  \rset^d$, $\norm{x} \geq \Rqtransition$
  \begin{equation}
\label{eq:def_Rq}
\inf_{y \in B(x)} \frac{q(y,x)}{q(x,y)}  \leq \eta/4 \eqsp.
    \end{equation}
    Now we are able to build the sequence $\{ x_n , n \in \nset\}$ such that for all $n
    \in \nset$, $\norm{x_{n+1}} \geq \max( \Repsilon, \Rqtransition)$ and $\lim_{n \to \plusinfty} \pi(x_n) =
    \plusinfty$. Indeed let $x_0 \in \rset^d$ such that $\norm{x_0} \geq \max( \Repsilon,
    \Rqtransition)$. Assume, we have built the sequence up to the $n$th term and such that for all
    $k=0, \dots, n-1$, $\norm{x_{k+1}} \geq \max( \Repsilon, \Rqtransition)$ and $\pi(x_{k+1}) \geq
    (3/2) \pi(x_k)$. Now we choose $x_{n+1}$ depending
    on $x_n$, satisfying $\pi(x_{n+1}) \geq (3/2) \pi(x_n)$ and $ \norm{x_{n+1}} \geq \max( \Repsilon, \Rqtransition)$. Since
    $\norm{x_n} \geq  \max( \Repsilon, \Rqtransition)$, by
    \eqref{eq:bound_rejection_proba}-\eqref{eq:def_M_non_geo} and \eqref{eq:def_Rq}
  \begin{align*}
\eta &\leq \int_{\rset^d} \alpha(x_n,y) q(x_n,y) \rmd y  \leq \eta/2 + \int_{B(x_n)}
\min \left(1, \frac{\pi(y)q(y,x_n)}{\pi(x_n) q(x_n,y)} \right)  q(x_n,y) \rmd y \\
& \leq \eta/2 + ( \eta/4)\int_{B(x_n)}\frac{\pi(y)}{\pi(x_n)} q(x_n,y) \rmd y \eqsp.
    \end{align*}
    This inequality implies that $\int_{B(x_n)}\frac{\pi(y)}{\pi(x_n)} q(x_n,y) \rmd y  \geq
    2$ and therefore there exists $x_{n+1} \in B(x_n)$ such that $\pi(x_{n+1})
    \geq (3/2)\pi(x_n)$, and since $x_{n+1} \in B(x_n)$ by \eqref{eq:norm_propo_greater},
    $\norm{x_{n+1}} \geq  \max( \Repsilon, \Rqtransition)$. Therefore, we have a sequence $\{x_n ,
    n \in \nset \}$ such that for all $n \in \nset$, $ \pi(x_{n+1}) \geq (3/2)\pi(x_n)$. Since
    by assumption $\pi(x_0) >0$, we get $\lim_{n \to \plusinfty} \pi(x_n) = \plusinfty$,
    which contradicts the assumption that $\pi$ is bounded.
This concludes the proof  of Theorem~\ref{theo:non_geo_erg_Gaussian_MH}.\qed

\section{Expressions of $C_5 ^\bullet(x_1,\xi_1)$}
\label{app:expression_C5}
\begin{multline*}
C_5^{\mMALA}(x_1,\xi_1) = \frac{\ell^5}{720} \left(\xi_1^5 g^{(5)}(x_1)+5 \xi_1^3 g^{(5)}(x_1)+15 \xi_1^3 g^{(4)}(x_1) g'(x_1) \right.\\
+15 \xi_1 g^{(4)}(x_1) g'(x_1)+30 \xi_1^3 g^{(3)}(x_1) g''(x_1)\\
\left. +10 \xi_1 g^{(3)}(x_1) g''(x_1)+30 \xi_1 g^{(3)}(x_1) g'(x_1)^2+35 \xi_1 g'(x_1) g''(x_1)^2\right)
\end{multline*}

  \begin{multline*}
      C_5^{\mO}(x_1,\xi_1) = \ell^5 \left( \frac{1}{720} \xi_1^5 g^{(5)}(x_1)+\frac{1}{144}
    \xi_1^3 g^{(5)}(x_1)+ \right.
     \frac{1}{48} \xi_1^3 g^{(4)}(x_1)
    g'(x_1)\\
    +\frac{1}{48} \xi_1 g^{(4)}(x_1) g'(x_1)
     +\frac{29}{144}
    \xi_1^3 g^{(3)}(x_1) g''(x_1) - \frac{7}{48} \xi_1
    g^{(3)}(x_1) g''(x_1) \\
    +\frac{1}{24} \xi_1 g^{(3)}(x_1)
    g'(x_1)^2+ \left. \frac{1}{6} \xi_1 g'(x_1) g''(x_1)^2 \right) \eqsp.
\end{multline*}

\begin{multline*}
  C_5^{\bO}(x_1,\xi_1) =  \ell^5 \left( \frac{1}{720} \xi_1^5 g^{(5)}(x_1)+\frac{1}{144}
  \xi_1^3 g^{(5)}(x_1)+\frac{1}{48} \xi_1^3 g^{(4)}(x_1)
  g'(x_1) \right.  \\
+\frac{1}{48} \xi_1 g^{(4)}(x_1) g'(x_1) +\frac{29}{144}
  \xi_1^3 g^{(3)}(x_1) g''(x_1)-\frac{19}{144} \xi_1
  g^{(3)}(x_1) g''(x_1)  \\
\left.   +\frac{1}{24} \xi_1 g^{(3)}(x_1)
  g'(x_1)^2+\frac{1}{6} \xi_1 g'(x_1) g''(x_1)^2 \right) \eqsp.
\end{multline*}

\begin{multline*}
   C_5^{\gbO}(x_1,\xi_1) =  \ell^5 \left( \frac{1}{720} \xi_1^5 g^{(5)}(x_1)+\frac{1}{144}
  \xi_1^3 g^{(5)}(x_1)+\frac{1}{48} \xi_1^3 g^{(4)}(x_1)
  g'(x_1) \right. \\
  +\frac{1}{48} \xi_1 g^{(4)}(x_1) g'(x_1)+\frac{1}{72}
  a_3 \xi_1 g^{(3)}(x_1) g''(x_1)+\frac{1}{6} a_4^2
  \xi_1^3 g^{(3)}(x_1) g''(x_1)\\
  -\frac{1}{6} a_4^2 \xi_1
  g^{(3)}(x_1) g''(x_1)+\frac{5}{144} \xi_1^3 g^{(3)}(x_1)
  g''(x_1)+\frac{1}{48} \xi_1 g^{(3)}(x_1) g''(x_1)\\
   +\frac{1}{24}
  \xi_1 g^{(3)}(x_1) g'(x_1)^2-\frac{1}{24} a_1^2 \xi_1
  g'(x_1) g''(x_1)^2  +\frac{1}{6} a_4^2 \xi_1 g'(x_1)
   g''(x_1)^2 \\
   \left. +\frac{1}{24} \xi_1 g'(x_1) g''(x_1)^2  \right) \eqsp.
\end{multline*}


\section{Expressions of $K^{\bullet}$}
\allowdisplaybreaks

\label{app:expression_K}
We provide here the expressions of the quantities $K^{\bullet}$ involved in Theorems \ref{theo:accept_ratio_fMALA}, \ref{theo:scaling:fMALA}, \ref{theo:accept_ratio_gbO}, \ref{theo:scaling:gbO}.
Let $\xrm$ be a random variable distributed according to $\pi_1$.
\begin{align*}
&K^{\mMALA} =\mathbb{E} \left[ \frac{79 g^{(5)}(\xrm)^2}{17280}+\frac{11 g^{(4)}(\xrm)^2 g'(\xrm)^2}{1152} +\frac{77 g^{(3)}(\xrm)^2 g''(\xrm)^2}{2592}+\frac{1}{576} g^{(3)}(\xrm)^2 g'(\xrm)^4 \right. \\
&+\frac{49 g'(\xrm)^2 g''(\xrm)^4}{20736}+\frac{7}{576} g^{(4)}(\xrm) g^{(5)}(\xrm) g'(\xrm)+\frac{19}{864} g^{(3)}(\xrm) g^{(5)}(\xrm) g''(\xrm) \\
&+\frac{1}{288} g^{(3)}(\xrm) g^{(5)}(\xrm) g'(\xrm)^2+\frac{7 g^{(5)}(\xrm) g'(\xrm) g''(\xrm)^2}{1728}\\
&+\frac{1}{144} g^{(3)}(\xrm) g^{(4)}(\xrm) g'(\xrm)^3+\frac{7}{864} g^{(4)}(\xrm) g'(\xrm)^2 g''(\xrm)^2+\frac{7 g^{(3)}(\xrm) g'(\xrm)^3 g''(\xrm)^2}{1728}\\
&\left. +\frac{5}{432} g^{(3)}(\xrm)^2 g'(\xrm)^2 g''(\xrm) +\frac{35 g^{(3)}(\xrm) g'(\xrm) g''(\xrm)^3}{2592}+\frac{29}{864} g^{(3)}(\xrm) g^{(4)}(\xrm) g'(\xrm) g''(\xrm) \right] \eqsp.
\end{align*}

\begin{align*}
  &K^{\mO} = \mathbb{E} \left[ \frac{79 g^{(5)}(\xrm)^2}{17280} + \frac{11
      g^{(4)}(\xrm)^2 g'(\xrm)^2}{1152}+\frac{1567 g^{(3)}(\xrm)^2 g''(\xrm)^2}{3456} \right. \\
  & + \frac{1}{576} g^{(3)}(\xrm)^2 g'(\xrm)^4+\frac{1}{36} g'(\xrm)^2
  g''(\xrm)^4+\frac{7}{576} g^{(4)}(\xrm) g^{(5)}(\xrm) g'(\xrm) \\
  &+\frac{17}{192} g^{(3)}(\xrm) g^{(5)}(\xrm) g''(\xrm)+\frac{1}{288} g^{(3)}(\xrm) g^{(5)}(\xrm)
  g'(\xrm)^2 \\
  &+\frac{1}{72} g^{(5)}(\xrm) g'(\xrm)
  g''(\xrm)^2+\frac{1}{144} g^{(3)}(\xrm) g^{(4)}(\xrm) g'(\xrm)^3 + \\
  &\frac{1}{36} g^{(4)}(\xrm) g'(\xrm)^2 g''(\xrm)^2+\frac{1}{72} g^{(3)}(\xrm) g'(\xrm)^3 g''(\xrm)^2 \\
  &+\frac{11}{288} g^{(3)}(\xrm)^2 g'(\xrm)^2 g''(\xrm) +\frac{11}{72} g^{(3)}(\xrm) g'(\xrm)
  g''(\xrm)^3+ \\
  & \left. \frac{73}{576} g^{(3)}(\xrm) g^{(4)}(\xrm) g'(\xrm) g''(\xrm) \right] \eqsp.
  \end{align*}
%

\begin{align*}
 &K^{\gbO} =
  \mathbb{E} \left[\frac{1}{36} g'(\xrm)^2 g''(\xrm)^4 a_4^4+\frac{5}{18} g''(\xrm)^2
  g^{(3)}(\xrm)^2 a_4^4 \right. \\
  &+\frac{1}{9} g'(\xrm) g''(\xrm)^3
  g^{(3)}(\xrm) a_4^4-\frac{1}{72} a_1^2 g'(\xrm)^2 g''(\xrm)^4
  a_4^2+\frac{1}{72} g'(\xrm)^2 g''(\xrm)^4 a_4^2\\
  &+\frac{11}{72}
  g''(\xrm)^2 g^{(3)}(\xrm)^2 a_4^2+\frac{1}{108} a_3
  g''(\xrm)^2 g^{(3)}(\xrm)^2 a_4^2\\
  &+\frac{1}{36} g'(\xrm)^2
  g''(\xrm) g^{(3)}(\xrm)^2 a_4^2-\frac{1}{36} a_1^2 g'(\xrm)
  g''(\xrm)^3 g^{(3)}(\xrm) a_4^2\\
  &+\frac{5}{72} g'(\xrm)
  g''(\xrm)^3 g^{(3)}(\xrm) a_4^2+\frac{1}{216} a_3 g'(\xrm)
  g''(\xrm)^3 g^{(3)}(\xrm) a_4^2 \\
  &+\frac{1}{72} g'(\xrm)^3
  g''(\xrm)^2 g^{(3)}(\xrm) a_4^2+\frac{1}{36} g'(\xrm)^2
  g''(\xrm)^2 g^{(4)}(\xrm) a_4^2\\
  &+\frac{7}{72} g'(\xrm)
  g''(\xrm) g^{(3)}(\xrm) g^{(4)}(\xrm) a_4^2+\frac{1}{72}
  g'(\xrm) g''(\xrm)^2 g^{(5)}(\xrm) a_4^2\\
  &+\frac{5}{72}
  g''(\xrm) g^{(3)}(\xrm) g^{(5)}(\xrm) a_4^2+\frac{1}{576}
  a_1^4 g'(\xrm)^2 g''(\xrm)^4\\
  &-\frac{1}{288} a_1^2 g'(\xrm)^2
  g''(\xrm)^4+\frac{1}{576} g'(\xrm)^2 g''(\xrm)^4+\frac{1}{576}
  g'(\xrm)^4 g^{(3)}(\xrm)^2\\
  &+\frac{a_3^2 g''(\xrm)^2
    g^{(3)}(\xrm)^2}{5184}
  +\frac{1}{288} a_3 g''(\xrm)^2
  g^{(3)}(\xrm)^2\\
  &+\frac{79 g''(\xrm)^2 g^{(3)}(\xrm)^2}{3456}+\frac{1}{96}
  g'(\xrm)^2 g''(\xrm) g^{(3)}(\xrm)^2\\
  &+\frac{1}{864} a_3
  g'(\xrm)^2 g''(\xrm) g^{(3)}(\xrm)^2+\frac{11 g'(\xrm)^2
    g^{(4)}(\xrm)^2}{1152}\\
  &+\frac{79 g^{(5)}(\xrm)^2}{17280}-\frac{1}{96}
  a_1^2 g'(\xrm) g''(\xrm)^3 g^{(3)}(\xrm)\\
  &+\frac{1}{96} g'(\xrm)
  g''(\xrm)^3 g^{(3)}(\xrm)-\frac{1}{864} a_1^2 a_3 g'(\xrm)
  g''(\xrm)^3 g^{(3)}(\xrm)\\
  &+\frac{1}{864} a_3 g'(\xrm)
  g''(\xrm)^3 g^{(3)}(\xrm)-\frac{1}{288} a_1^2 g'(\xrm)^3
  g''(\xrm)^2 g^{(3)}(\xrm)\\
  &+\frac{1}{288} g'(\xrm)^3 g''(\xrm)^2
  g^{(3)}(\xrm)-\frac{1}{144} a_1^2 g'(\xrm)^2 g''(\xrm)^2
  g^{(4)}(\xrm)\\
  &+\frac{1}{144} g'(\xrm)^2 g''(\xrm)^2
  g^{(4)}(\xrm)+\frac{1}{144} g'(\xrm)^3 g^{(3)}(\xrm)
  g^{(4)}(\xrm)\\
  &+\frac{17}{576} g'(\xrm) g''(\xrm) g^{(3)}(\xrm)
  g^{(4)}(\xrm)+\frac{1}{432} a_3 g'(\xrm) g''(\xrm)
  g^{(3)}(\xrm) g^{(4)}(\xrm)\\
  &-\frac{1}{288} a_1^2 g'(\xrm)
  g''(\xrm)^2 g^{(5)}(\xrm)+\frac{1}{288} g'(\xrm) g''(\xrm)^2
  g^{(5)}(\xrm)+\\
  &\frac{1}{288} g'(\xrm)^2 g^{(3)}(\xrm)
  g^{(5)}(\xrm)+\frac{11}{576} g''(\xrm) g^{(3)}(\xrm)
  g^{(5)}(\xrm)\\
  &\left. +\frac{1}{864} a_3 g''(\xrm) g^{(3)}(\xrm)
  g^{(5)}(\xrm)+\frac{7}{576} g'(\xrm) g^{(4)}(\xrm) g^{(5)}(\xrm) \right] \eqsp.
    \end{align*}


\end{document}